\documentclass[10pt]{amsart}

\usepackage{amsmath, amsthm, amssymb, graphicx, enumitem, everypage, datetime, getfiledate} 
\usepackage{wrapfig}

\newcounter{citedtheorems}

\newtheorem{defn}{Definition}[section]
\newtheorem{theorem}[defn]{Theorem}

\newtheorem*{theorem-m}{Theorem \ref{main-theorem}}
\newtheorem*{thm-p2a}{Theorem \ref{t:p2a}}

\newtheorem*{thm-seq}{Theorem \ref{t:seq}}
\newtheorem*{thm-e}{Theorem}
\newtheorem*{thm-m}{Main Theorem}
\newtheorem*{theorem-abs1}{Theorem \ref{ind-theorem}}
\newtheorem*{theorem-abs2}{Theorem \ref{a23}}
\newtheorem*{theorem-abs3}{Theorem \ref{ind-new}}
\newtheorem*{theorem-abs4}{Theorem \ref{m1}}
\newtheorem*{thm-x}{Theorem}
\newtheorem{thm-lit}[citedtheorems]{Theorem}
\newtheorem{defn-lit}[citedtheorems]{Definition}
\newtheorem{fact-lit}[citedtheorems]{Fact}
\newtheorem{fact}[defn]{Fact}

\newtheorem{defn-claim}[defn]{Definition/Claim}

\newtheorem*{defn-in}{Definition \arabic{section}.\arabic{equation}}

\newtheorem*{claim-in}{Claim \arabic{section}.\arabic{equation}}

\newtheorem{conv}[defn]{Convention}
\newtheorem{claim}[defn]{Claim}
\newtheorem{subclaim}[defn]{Subclaim}
\newtheorem{lemma}[defn]{Lemma}
\newtheorem{obs}[defn]{Observation}

\newtheorem{comm}[defn]{Comment}

\newtheorem{ntn}[defn]{Notation}
\newtheorem{disc}[defn]{Discussion}

\newtheorem{qst}[defn]{Question}

\newcommand{\br}{\vspace{2mm}}
\newcommand{\sbr}{\vspace{1mm}}

\newcommand{\eff}{\mathcal{F}}
\newcommand{\gee}{\mathcal{G}}

\newcommand{\ml}{\mathcal{L}}
\newcommand{\tlf}{\trianglelefteq}
\newcommand{\tlfn}{\triangleleft}

\newcommand{\rn}{\operatorname{range}}
\newcommand{\cf}{\operatorname{cof}}

\newcommand{\dom}{\operatorname{dom}}

\newcommand{\mcv}{\mathcal{V}}

\newcommand{\uu}{\mathcal{U}}
\newcommand{\mlx}{\mathcal{M}_{\lambda}}
\newcommand{{\xw}}{\mathbf{w}}

\newcommand{\xn}{\mathfrak{n}}

\newcommand{\vrt}{\operatorname{vert}}

\newcommand{\mcf}{\mathcal{F}}

\newcommand{\tp}{\operatorname{tp}}

\newcommand{\supp}{\operatorname{supp}}

\newcommand{\mct}{\mathcal{T}}

\newcommand{\zm}{\mathcal{M}}
\newcommand{\lgn}{\operatorname{lg}}

\newcommand{\de}{\mathcal{D}}

\newcommand{\ts}{\mathbf{S}}

\newcommand{\jj}{\mathbf{j}}
\newcommand{\mcp}{\mathcal{P}}

\newcommand{\ba}{\mathfrak{B}}

\newcommand{\rstr}{\upharpoonright}

\newcommand{\mcr}{\mathcal{R}}
\newcommand{\bad}{\mathcal{P}}

\newcommand{\mcs}{\mathcal{S}}

\newcommand{\vp}{\varphi}

\newcommand{\ma}{\mathbf{a}}
\newcommand{\mb}{\mathbf{b}}
\newcommand{\mc}{\mathbf{c}}

\newcommand{\mx}{\mathbf{x}}

\newcommand{\fin}{\operatorname{FI}}

\newcommand{\trg}{T_{\mathbf{rg}}}

\newcommand{\xm}{\mathfrak{m}}
\newcommand{\xr}{\mathfrak{r}}

\newcommand{\clm}{\operatorname{cl}_{\zm}}

\newcommand{\tsf}{T_f}
\newcommand{\tsfo}{T_{0,f}}

\newcommand{\leaf}{\operatorname{leaf}}

\title{An example of a new simple theory}

\author{M. Malliaris and S. Shelah}
\thanks{\emph{Thanks:} Malliaris was partially supported by DMS-1553653 and by a Minerva Research Foundation membership at the IAS. 
Shelah was partially supported by European Research Council grant 338821.   Both authors thank 
NSF grant 1362974 (Rutgers) and ERC 338821. 
This is paper 1140 in Shelah's list.}

\address{Department of Mathematics, University of Chicago, 5734 S. University Avenue, Chicago, IL 60637, USA} 
\email{mem@math.uchicago.edu}

\address{Einstein Institute of Mathematics, Edmond J. Safra Campus, Givat Ram, The Hebrew
University of Jerusalem, Jerusalem, 91904, Israel, and Department of Mathematics,
Hill Center - Busch Campus, Rutgers, The State University of New Jersey, 110
Frelinghuysen Road, Piscataway, NJ 08854-8019 USA}
\email{shelah@math.huji.ac.il}
\urladdr{http://shelah.logic.at}

\begin{document}

\begin{abstract} 
We construct a countable simple theory which, in Keisler's order, is strictly above the random graph  
and also in some sense orthogonal to the building blocks of the recently discovered infinite descending chain. As a result we prove in ZFC that there are incomparable classes in Keisler's order. 
\end{abstract}

\maketitle

\br
%\hfill{\emph{For Simon Thomas on the occasion of his birthday.}}

\br

Recent work on the structure of Keisler's order is changing our understanding of the 
so-called simple unstable theories, a class which includes the random graph and pseudofinite fields. 
Indeed, when we discovered recently \cite{MiSh:1050} that Keisler's order has infinitely many classes -- overturning a long-standing idea that 
it had five or six -- these infinitely many classes were within the simple theories.  We are starting to see that the simple unstable theories may have 
very interesting layers of complexity above that of the random graph, arising from the interaction of randomness with underlying constraints, with no obvious analogue in the stable case. 

Basic questions remain open about the structure of Keisler's order on the simple theories, 
in part because of a lack of examples. 
For instance, among the unstable theories, the Keisler-minimum class is the class of the random graph. 
The infinitely many classes arise from disjoint unions of certain random hypergraphs with 
forbidden substructures  (higher analogues of the triangle-free random graph \cite{h:letter}) 
forming a descending chain above the random graph. It was not known whether there were 
theories strictly between the random graph and this infinite descending chain. 

In the present paper, we build a new simple theory, really  a family of theories, illustrating how 
the randomness may interact with an underlying structure with some forbidden configurations, 
even with only a graph edge (no hyperedges required).  We analyze its saturation and non-saturation in regular 
ultrapowers, concluding it is strictly above the random 
graph but in a precise sense orthogonal to the higher analogues of the triangle-free random graph.  
As a consequence of our strategy,  we prove there are incomparable classes in Keisler's order just in ZFC; 
this was known to be true under the existence of a supercompact cardinal, 
as noticed independently by Ulrich \cite{ulrich} and the authors \cite{MiSh:F1530}. 
Along the way we give a gentle introduction to some key methods of \cite{MiSh:999}, \cite{MiSh:1030}, \cite{MiSh:1050}, 
the papers which have been the foundation of our current work on simple theories.  

We thank the anonymous referee and D. Ulrich for very helpful comments on the manuscript.  

 \newpage

\tableofcontents

\section{Notation and preliminaries}

Keisler's order compares complete countable theories via the difficulty of saturating their regular ultrapowers. 

\begin{defn}[Keisler's order, 1967 \cite{keisler}]  \label{d:keisler}
Let $T_1, T_2$ be complete countable theories. We say $T_1 \tlf T_2$ if for every infinite $\lambda$, every regular ultrafilter 
$\de$ on $\lambda$, every model $M_1 \models T_1$, every model $M_2 \models T_2$,   ~ if $(M_2)^\lambda/\de$ is $\lambda^+$-saturated, 
then $(M_1)^\lambda/\de$ is $\lambda^+$-saturated. 
\end{defn}

We remind the reader of Keisler's result that if $\de$ is a regular ultrafilter on $\lambda$ and $M \equiv N$ in a countable language, then 
$M^\lambda/\de$ is $\lambda^+$-saturated iff $N^\lambda/\de$ is $\lambda^+$-saturated. \emph{Thus}, the choice of $M_1$ and $M_2$ in 
Definition \ref{d:keisler} is only important up to elementary equivalence. 
A further introduction to Keisler's order can be found in the recent lecture notes \cite{AST} sections 2-3 and 
in sections 1-2 of \cite{MiSh:1030}. 

\begin{conv} \emph{ } In what follows: 
\begin{enumerate}
\item All theories are complete and countable, unless otherwise stated. 
\item When $\de$ is a regular ultrafilter on $I$ and $T$ is a theory, we will say ``$\de$ is good for $T$'' to mean that for some 
$($equivalently every$)$ model $M \models T$, the ultrapower $M^I/\de$ is $|I|^+$-saturated. 
\end{enumerate}
\end{conv}

Our recent work on simple theories started with an idea in \cite{MiSh:999} to increase the range of ultrafilter construction. The idea is to build regular 
ultrafilters in two stages: by building a regular filter $\de_0$ on $I$ so that the quotient Boolean algebra $\mcp(I)/\de_0$ is 
isomorphic to some specific Boolean algebra $\ba$, and then building an ultrafilter $\de_*$ on $\ba$ (which, a priori, need not be regular), 
finally combining $\de_0$ and $\de_*$ in the natural way to obtain an ultrafilter $\de$ on $I$. 

A key lemma, ``separation of variables,'' says that in this setup there is a natural translation between realizing types in the ultrapower $M^I/\de$ and 
showing that certain related patterns, called ``possibility patterns,'' have multiplicative refinements in $\de_*$ (see below).   As a result, 
in many subsequent saturation-of-ultrapowers arguments it is most convenient to work in the 
Boolean algebra $\ba$, a completion of a free boolean algebra. 

\begin{defn} \label{d:ba} $\ba^0_{2^\lambda,  \mu, \theta}$ is the free Boolean algebra generated by $2^\lambda$ independent partitions 
each of size $\mu$, where intersections of $<\theta$ elements of distinct partitions are nonempty. 
For a boolean algebra $\ba$, let $\ba^1$ denote its completion, and let $\ba^+$ denote its nonzero elements.  
\end{defn}

\begin{fact} \label{fact-iff}
Assuming $\mu \leq \lambda = \lambda^{<\theta}$, we have that 
$\ba^0_{2^\lambda, \mu, \theta}$ and thus its completion exists. 
In this paper,  to find $D_0, \de_*, \jj$ as in \ref{d:built} below, we use the completion. 
\end{fact}

\begin{proof}
See Engelking-Karlowicz \cite{ek}, Fichtenholz and Kantorovich\cite{f-k}, Hausdorff \cite{hausdorff}, 
or Shelah \cite{Sh:c} Appendix, Theorem 1.5. 
\end{proof}

The following discussion may help the picture. Consider $\ba^0_{2^\lambda, \mu, \theta}$ as generated by 
$2^\lambda$ independent antichains each of size $\mu$, which we may enumerate as:
\[ \langle \mx_{\alpha, \epsilon} : \alpha < 2^\lambda, \epsilon < \mu \rangle. \]
For each fixed $\alpha$, $\langle \mx_{\alpha, \epsilon} : \epsilon < \mu \rangle$ is an antichain, so 
\[ \ma_{\alpha, \epsilon} \cap \ma_{\alpha, \epsilon^\prime} = 0 \mbox{ for } \epsilon < \epsilon^\prime < \mu.  \]
Otherwise, the intersection of any $<\theta$ elements no two of which are on the same antichain is nonzero.   
A compact way to denote such an intersection, say of 
\begin{equation}
\label{e:basis} 
 \bigcap_{\ell < \beta}  \mx_{\alpha_{i_\ell} } 
\end{equation}
where $\beta < \theta$ and the sequence $\langle \alpha_{i_\ell} : \ell < \beta \rangle$ is without repetition, is as follows. 
Consider the function $f = \{ (\alpha_{i_\ell}, \epsilon_{i_\ell}) : \ell < \beta \}$ whose domain is an element of 
$[2^\lambda]^{<\theta}$ and which assigns to each element of its domain (the index for an antichain) a value less than $\mu$ 
(the specific element of the antichain). Then let 
\begin{equation}
\label{e:first} \mbox{``}\mx_f \mbox{''}=  \bigcap_{\ell < \beta}  \mx_{\alpha_{i_\ell}} > 0. 
\end{equation}
Now consider the completion $\ba^1_{2^\lambda, \mu, \theta}$. As the generators are dense in the completion, 
for every nonzero $\mc$, there is some nonzero element $\leq \mc$ which is an intersection of the form (\ref{e:basis}).

For the present paper, the reader could forget the quotes around $\mx_f$ without much loss, but 
we should also explain the general notation. 
In \ref{fact-iff} above, assuming $\mu \leq \lambda = \lambda^{<\theta}$, what one proves is the existence of a family of functions 
$\gee \subseteq {^\lambda \mu}$ of size $2^\lambda$, let us write it as 
$\{ g_\alpha : \alpha < 2^\lambda \}$, which is \emph{independent} in the following sense.  Given any distinct 
$\langle \alpha_{i_\ell} : \ell < \beta < \theta \rangle$ and any values $\langle \epsilon_{i_\ell} : \ell < \beta \rangle$ such that 
for each $\ell < \beta$, $i_\ell$ is in the range of $g_{\alpha_{i_\ell}}$, we have that the set 
\begin{equation} \label{e:antichain}
\{ t < \lambda : g_{\alpha_{i_\ell}}(t) = \epsilon_{i_\ell} \mbox{ for all $\ell < \beta$} \} \neq \emptyset. 
\end{equation}
In this setup, each function $g \in \gee$ has domain $\lambda$ and range $\mu$ and so naturally gives rise to a partition of $\lambda$ into $\mu$ pieces, which we can represent as 
an antichain as above.  Then the intersection (\ref{e:antichain}) is naturally described by a function, call it $f$, 
which assigns $g_{\alpha_{i_\ell}}$ to $\epsilon_{i_\ell}$ for each $\ell < \theta$, so $\mx_f$  is just the nonempty set (intersection) in 
(\ref{e:antichain}).  
So indeed (\ref{e:antichain}) is parallel to (\ref{e:first}),  the only difference is whether we take the domain of $f$ to be 
the set of indices for the antichains or, more generally, the set of functions giving them.  This explains the definition \ref{d:fin} in the general case. 

We summarize this discussion with the following notation:

\begin{defn}  \label{d:fin} Let 
\[ \fin_{\mu,\theta}(2^\lambda) = \{ h : h \mbox{ is a function,} \dom(h) \subseteq 2^\lambda, \rn(h) \subseteq \mu, |\dom(h)| < \theta \}. \]
For $g \in \fin_{\mu,\theta}(2^\lambda)$, let $\mx_g$ denote the corresponding nonzero element of $\ba$. 
\end{defn}

\begin{conv} 
We will assume that giving $\ba = \ba^1_{\alpha, \mu, \theta}$ determines $\alpha, \mu, \theta$ and a set of generators 
$\langle \mx_f : f \in \fin_{\mu,\theta}(\alpha) \rangle$. 
\end{conv}

Note that for any $f \in \fin_{\mu,\theta}(2^\lambda)$, $\mx_f > 0$. 
As the generators are dense in the completion, we will often apply this in the form  ``as $\mc \in (\ba^1_{2^\lambda, \mu, \theta})^+$, we may choose $g \in \fin_{\mu,\theta}(2^\lambda)$ so that $\mx_g \leq \mc$.''
$\ba$ is generated freely by its generators except for the following conditions: first,  
if $f \subseteq f^\prime  \in \fin_{\mu,\theta}(2^\lambda)$, then 
$\mx_f \geq \mx_{f^\prime}$, because we are adding more conditions to the intersection. Second, 
for any $f, f^\prime \in \fin_{\mu,\theta}(2^\lambda)$, $\mx_f \cap \mx_{f^\prime} > 0$ if and only if $f, f^\prime$ agree on their 
common domain (this includes the case where their common domain is empty).

\br

\noindent Next we recall the setup of separation of variables: 

\begin{defn}[Regular ultrafilters built from tuples, from \cite{MiSh:999} Theorem 6.13] \label{d:built}
Suppose $\de$ is a regular ultrafilter on $I$, $|I| = \lambda$. We say that $\de$ is built from 
$(\de_0, \ba, \de_*)$ when: 

\begin{enumerate}
\item {$\de_0$ is a regular, $|I|^+$-excellent filter on $I$} 
\\ {$($for our purposes here, it is sufficient to use regular and good$)$}
\item {$\ba$ is a Boolean algebra}
\item {$\de_*$ is an ultrafilter on $\ba$}
\item {there exists a surjective homomorphism $\jj : \mcp(I) \rightarrow \ba$ such that:}
\begin{enumerate}
\item $\de_0 = \jj^{-1}(\{ 1_\ba \})$ 
\item $\de = \{ A \subseteq I : \jj(A) \in \de_* \}$.
\end{enumerate}
\end{enumerate}
\end{defn}

\begin{thm-lit}[``Separation of variables'', \cite{MiSh:999} Theorem 6.13] \label{t:separation}
Suppose that $\de$ is a regular ultrafilter on $I$ built from $(\de_0, \ba, \de_*)$. Then the following are equivalent: 
\begin{itemize}
\item[(A)] $\de_*$ is $(|I|, \ba, T)$-moral, see \cite{MiSh:999} Definition $6.3$. 
\item[(B)] $\de$ is good for $T$. 
\end{itemize}
\end{thm-lit}

The definition of ``moral'' is a bit long to quote, but can be easily summarized by saying that the way we will use Theorem \ref{t:separation} 
is the following. Suppose $\de$ is a regular ultrafilter on $I$, $|I| = \lambda$, $\de$ is built from $(\de_0, \ba, \de_*)$, $M \models T$, and 
$N = M^I/\de$ is the ultrapower. Suppose $p = \{ \vp_\alpha(x,\bar{a}_\alpha) : \alpha < \lambda \}$  is a type in the ultrapower. 
Let $\bar{B} = \langle B_u : u \in [\lambda]^{<\aleph_0} \rangle$ be a sequence of elements of $\de$ such that for each finite $u$, 
\[ B_u = \{ t \in I : M \models \exists x \bigwedge_{\alpha \in u} \vp_\alpha(x,\bar{a}_\alpha[t]) \}. \]
Let $\bar{\mb} = \langle \mb_u : u \in [\lambda]^{<\aleph_0} \rangle$ be a sequence of elements of $\de_*$ such that 
$\jj(B_u) = \mb_u$ for each $u$.  Then to show $p$ is realized, it will suffice to show that $\bar{\mb}$ has a multiplicative 
refinement in $\ba$, i.e. there exists a sequence $\langle \mb^\prime_{u} : u \in [\lambda]^{<\aleph_0} \rangle$ of elements 
of $\de_*$ such that for each finite $u$, 
$\mb^\prime_u \leq \mb_u$, and such that for each finite $u$ and $v$, $\mb^\prime_u \cap \mb^\prime_v = 
\mb^\prime_{u \cup v}$.

\br
The next claim verifies that we can generally extract many compatible elements from a large enough subset in our Boolean algebra. 

\begin{claim} \label{intersections} Let $\ba = \ba^1_{2^\lambda, \mu, \aleph_0}$. 
Suppose $\langle g_\alpha : \alpha \in \uu \rangle$ is a subset of $\fin_{\mu,\aleph_0}(2^\lambda)$, where $\uu$ is of cardinality $>\mu$, and 
$m <\omega$.  Then there is $u \subseteq \uu$, $|u| = m$ such that $\bigcup_{\alpha \in u} g_\alpha$ is a function. 
\end{claim}

\begin{proof} 
By induction on $n \leq m$ we will choose $\beta_n$,  
$\uu_n$, and $\langle f^n_\alpha : \alpha \in \uu_n \rangle$ such that: 
\begin{enumerate}
\item $\beta_n \in \uu_{n-1}$, where $\uu_{-1} = \uu$
\item $\uu_n \subseteq \uu$ is of cardinality $>\mu$
\item $f^n_\alpha \in \fin_{\mu,\aleph_0}(2^\lambda)$ and $f^n_\alpha \supseteq g_\alpha$, for all $\alpha \in \uu$
\item $n < n^\prime$ implies $f^n_\alpha \subseteq f^{n^\prime}_\alpha$
\item $n < n^\prime$ implies $\uu_{n^\prime} \subseteq \uu_n$
\item for every $\alpha \in \uu_n$, $\mx_{f^n_\alpha} \leq \mx_{f_{\beta_n}}$ 
\item $\beta_n \notin \uu_n$  
\end{enumerate}
We will repeatedly use the fact that no set of size $>\mu$ is covered by a union of $\leq \mu$ sets of cardinality $\leq \mu$. 

At stage $-1$, $f^{-1}_\alpha = g_\alpha$, $\uu_{-1} = \uu$, and let ``$\mx_{f_{\beta_{-1}}}$'' mean $1_\ba$. 

At stage $n \geq 0$, by induction on $t$ we try to build a maximal antichain of $\mx_{f_{\alpha_{n-1}}}$'s, 
say $\langle \mx_{f_{\alpha_{i_t}}} : t < t_* \rangle$, where each $f_{\alpha_i} \supseteq f^{n-1}_{\alpha_i}$ for some $\alpha_i \in \uu_{n-1}$. 
This must stop at $t_*$ an ordinal $< \mu^+$ as $\ba$ has the $\mu^+$-c.c.   
For each remaining $\alpha \in \uu_{n-1}$, the element $\mx_{f^{n-1}_\alpha}$ must 
be compatible with one of the elements of our antichain (since the construction could not continue), and since $\uu_{n-1}$ is of cardinality 
$>\mu$, there is a subset $\uu_n \subseteq \uu_{n-1}$ of cardinality $>\mu$ 
whose elements are all $\neq \beta_n$ compatible with a single element of the antichain. Let $\beta_n$ be the subscript of 
this single element. For all $\alpha \in \uu_n$, let $f^n_\alpha = f^{n-1}_\alpha \cup f^{n-1}_{\beta_n}$.  (For $\alpha \in \uu \setminus \uu_n$, 
let $f^n_\alpha = f^{n-1}_\alpha$.) So we can carry the construction. 

\noindent When we arrive to $m$, 
we have defined $f^m_{\beta_0}, \dots, f^m_{\beta_m}$ and by construction, 
\[ 0 <  \mx_{f^m_{\beta_m}} \leq \cdots \leq \mx_{f^0_{\beta_0}}. \]
Recalling (3) and (4) of the inductive hypothesis, we conclude that  
\[ \mx_{g_{\beta_0}} \cap \cdots \cap \mx_{g_{\beta_m}} > 0 \]
i.e. when $u = \{ \beta_0, \dots, \beta_m \}$, $\bigcup_{\beta \in u} g_{\beta} $ is a function, as desired. 
(This proof is also easy by the $\Delta$-system lemma.)
\end{proof}

\begin{ntn} \label{ntn-f}
Let $\mcf = \{ f : f $ a {strictly} {increasing} function from $\mathbb{N}$ to $\mathbb{N} \setminus \{ 0 \}$$\}$. 
In particular, $f(k) > k$ for all $k \in \mathbb{N}$. 
\end{ntn}

\vspace{5mm}

\section{Construction}

In this section we build our example, really a family of examples taking as a parameter an increasing function $f \in \mcf$ (see \ref{ntn-f}). We begin with the definition, and continue with a leisurely discussion of the types in models of this theory. 

\begin{ntn} \label{orig-ntn}  Let $f \in \mcf$. 
\begin{enumerate}
\item[(a)] For each $k<\omega$, let $\mct_k = \mct_{f,k} = \{ \eta : \eta \in {^k \omega} $ and $\ell < k \rightarrow \eta(\ell) \leq f(\ell) \}$.

\sbr
\item[(b)] Let $\mct_{[k]} = \mct_{f, [k]} = \bigcup_{j \leq k } \mct_j$.

\sbr
\item[(c)] Let $\mct = \mct_f =  \bigcup_{k<\omega} \mct_k$. 

\sbr
\item[(d)] Let $\lim (\mct) = \{ \eta \in {^\omega \omega } : \eta \rstr k \in \mct_k \mbox{ for all } k<\omega \}$.

\sbr
\item[(e)] We may write $\mct, \mct_k$, etc. omitting $f$ when it is clear from context. 

\end{enumerate}
\end{ntn}

Notation \ref{orig-ntn} gives us trees $\mct_{k}$ of various finite heights whose splitting is finite but growing: for example, 
the root has two successors, nodes at level 1 each have three successors, nodes at level 2 each have 5 successors. (It isn't 
crucial that the function is strictly increasing, but it should not be eventually constant.)  $\mct$ is the tree whose 
restriction to height $k$ is $\mct_k$ for all finite $k$. 

\begin{defn} \label{d:nice} 
We say $s \subseteq \mct_{[k]}$ is \emph{$k$-maximal} if either $k=0$, or $k>0$ and $s$ satisfies: 
it is nonempty and closed under initial segment, 
every maximal node belongs to $\mct_k$, for no $\rho \in \mct_{[k]}$ do we have  
$ \{ \rho^\smallfrown \langle \ell \rangle : \ell \leq f(\lgn(\rho)) \} \subseteq s $,  \emph{ and } $s$ is maximal subject to these conditions. 
\end{defn}

Definition \ref{d:nice} describes a subtree of $\mct_{[k]}$ which is maximal subject to containing no full splitting.  
This means: choose all but one immediate successors of the root; for every element you have chosen at level 1, 
choose all but one of its immediate successors; and so on. (Once an element is \emph{not} chosen, all of its successors 
later in the tree are ruled out.)  There is a natural partial order on these subtrees, as the next 
definition points out.

\begin{ntn} \label{n:ds} Let $f \in \mcf$.
\begin{enumerate}
\item[(f)] Let $\mcs_k = \mcs_{f,k} = \{ s : s \subseteq \mct_{[k]} \mbox{ is $k$-maximal} \}$. 

\sbr

\item[(g)] Let $\mcs = \mcs_f = \bigcup_k \mcs_k$. 

\sbr
\item[(h)] Let $\lim (\mcs) = \{ \nu : \nu = \langle s_k : k<\omega \rangle $ is such that $s_k \in \mcs_k$ for each $k<\omega$ and 
$s_k \subseteq s_{k^\prime}$ for all $k < k^\prime < \omega \}$. 

\sbr
\item[(i)] We may write $\mcs, \mcs_k$, etc. omitting $f$ when it is clear from context. 
\end{enumerate}
\end{ntn}

To emphasize that $\lim (\mcs)$ and $\lim (\mct)$ are parallel, we might think of elements of $\lim (\mct)$ as sequences 
$\langle \eta_k : k <\omega \rangle$ where each $\eta_k \in \mct_k$ and $k<k^\prime < \omega$ implies $\eta_k \tlf \eta_{k^\prime}$.  
Elements of $\lim (\mct)$ are branches of $\mct$ under the ordering $\tlf$, and 
elements of $\lim (\mcs)$ are branches of $\mcs$ under the ordering $\subseteq$.  In particular, since all of the $\mct_k$'s and 
$\mcs_k$'s are finite and $f$ is increasing, both $\mct$ and $\mcs$ are countable, and both $\lim (\mct)$ and $\lim (\mcs)$ have cardinality 
$\leq 2^{\aleph_0}$. 

\br

We will define our theory $\tsf$ as the model completion of the following universal theory $\tsfo$.   
Note that this theory is \emph{not} in the same language as $\mct$ and $\mcs$. We simply use a 
family of unary predicates to partition the universe of the model into $P$ and $Q$ and then further partition these 
to mirror $\mct$ and $\mcs$. [For example, if $\mct$ has a root at level 0 followed by three nodes at level 1, 
$P$ is partitioned into one piece by a predicate $P_0$, which is in turn partitioned into three pieces by 
predicates called $P_{00}, P_{01}, P_{02}$, and so forth.]  So the tree structure on each side is 
``hard coded'' into the picture without the complexity of having a definable partial order in the 
language. The remaining ingredient is $R \subseteq P \times Q$, described after the definition. 

\begin{defn} \label{c2} Let $f \in \mcf$.  We define $\tsfo$ as the following universal theory. 

\begin{enumerate}
\item[(1)] $\tau(\tsfo) = \{ P, Q, R, {P}_\eta, Q_s : \eta \in \mct,  s \in \mcs  \} $
\\ where $R$ is a binary predicate and the $Q_s$'s, $P_\eta$'s, $P, Q$ are unary predicates.

\item[(2)] for a $\tau$-model $M$, $M \models \tsfo$ iff:
\begin{enumerate}
\item[(a)] $P^M$, $Q^M$ is a partition of $M$
\item[(b)] for each $k<\omega$, $\langle P^M_\eta : \eta \in \mct_k \rangle$ is a partition of $P^M$ 

\item[(c)] for each $k<\omega$, $\langle Q^M_s : s \in \mcs_k \rangle$ is a partition of $Q^M$
\item[(d)] $\eta \tlf \eta^\prime \in \mct$ implies $P^M_{\eta^\prime} \subseteq P^M_{\eta}$ and 
$s \subseteq s^\prime \in \mcs$ implies $Q^M_{s^\prime} \subseteq Q^M_s$
\item[(e)] $R^M \subseteq Q^M \times P^M$ satisfies:
\begin{itemize}
\item if $\eta \in \mct_k$, and $b_\ell \in P^M_{\eta^\smallfrown\langle \ell \rangle}$ for $\ell \leq f(k)$, 
then 
\[ M \models \neg (\exists x)(Q(x) \land \bigwedge_{\ell \leq f(k)} R(x,b_\ell)). \]   
\item if $(b,a) \in R^M$ and $b \in Q_s$, $s \in \mcs_k$, $a \in P_\eta$, $\eta \in \mct_k$ then $\eta \in s$.
\end{itemize}

\end{enumerate}
\item[(3)] Let $\tsf$ be the model completion of $\tsfo$. 
\end{enumerate}
\end{defn}

\noindent A more detailed informal description: 

$Q$ and $P$  partition the domain. 
If we think of the elements of $P$ as being indexed by the leaves of some infinite, finitely branching tree whose branching at height $k$ is $f(k)$, then 
for each $\eta \in \mct_k$, $P_\eta$ names the elements in the cone above $\eta$. In particular, for each finite $k$ and $\eta \in \mct_k$,  
$P^M_{\eta^\smallfrown \langle \ell \rangle}$ for $\ell \leq f(k)$ partition $P^M_\eta$.  

$R$, the edge relation, relates elements of $Q$ to elements of $P$. 
The restriction on $R$ is that for each element $b \in Q$ and each $k<\omega$ and $\eta \in \mct_k$, 
$b$ cannot connect to some element in each piece of the partition 
$\{ P^M_{\eta^\smallfrown \langle \ell \rangle} : \ell \leq f(k) \}$.    So for each $b \in Q^M$, each $k<\omega$ and each $\eta \in \mct_k$, 
either $R(b,x)$ misses the set $P^M_\eta$ entirely, or else it misses one of $P^M_{\eta^\smallfrown \langle \ell \rangle}$ for $\ell \leq f(k)$. 

If we think of the elements of $Q_s$ as being indexed by the elements of $\mcs$, in other words, 
by some subtree of finite height which is not full, then the elements of $Q_s$ are only allowed to connect via $R$ to elements in 
$P_\eta$ where $\eta \in s$.  
Notice that these predicates name already definable sets (and so will help with quantifier elimination): 
$Q^M_s(b)$ means 
\[ \{ \eta \in \mct_k : M \models (\exists x) (R(b,x) \land P_\eta(x)) \} = s \in\mcs_k.  \]
Of course we can't formally quantify over $\eta \in \mct_k$, but for each $k$, there are only 
finitely many $P_\eta$'s to consider, so the set named by $Q_s$ is indeed a definable set. Condition (2)(c) uses that the $s$ are maximal in the 
sense of \ref{d:nice}. 

Finally, note that in models of $\tsf$, the model completion, 
whenever $R(b,x)$ is consistent with $P_\eta(x)$, \emph{exactly} one of the formulas $\{ P_{\eta^\smallfrown \langle \ell \rangle}(x) : \ell \leq f(k) \}$ will be inconsistent with $R(b,x)$. So the $Q_s$'s with $s$ maximal are indeed the relevant ones.

\begin{defn}
For each $k < \omega$, let $T_{0,f,k}$ be the universal theory which is the restriction of $T_{0,f}$ to 
$\tau_k = \{ P, Q, R, {P}_\eta, Q_s : \eta \in \mct_k,  s \in \mcs_k  \} $.
\end{defn}

\begin{obs} \label{tof}
For each $k < \omega$,  $T_{0,f,k}$  is a universal theory in a finite relational language. It has the joint embedding property and amalgamation: 
given models $M, N$, the model 
whose domain is the disjoint union of $\dom(M) \cap \dom(N)$, $\dom(M) \setminus \dom(N)$, and $\dom(N) \setminus \dom(M)$ is also a model of the theory. As a result, a  
model completion $T_{f,k}$ exists and is well defined, and complete. 
\end{obs}

\begin{claim} \label{c5} $\tsfo$ has a model completion $\tsf$. In fact $\tsfo$ is relational, has amalgamation 
and the joint embedding property. 
Moreover $\tsf$ is simple, has elimination of quantifiers
has no algebraicity, and has trivial forking. 
\end{claim}

\begin{proof}[Proof Sketch] There are two parts. 

\br

\noindent\emph{Part I.}
First, we briefly describe the two key cases for quantifier elimination, ignoring the atomic formulas $x=a$. 
\br

(a) given $P(y_0), \dots, P(y_n)$, does there exist $x$ s.t. $Q(x) \land \bigwedge_{i\leq n} R(x,y_i)$ ? 

Informally, the answer is no if and only if the $y_i$'s fall in all pieces of some successor partition, that is, 
for some $k < \omega$ and $\eta \in \mct_k$, for each $\ell \leq f(k+1)$ there is at least one $ i  \leq n$ such 
that $P^M_{\eta^\smallfrown\langle \ell \rangle}(y_i)$. {Since $f$ is increasing} (\ref{ntn-f}), we 
can choose the minimal $k_*$ such that $f(k_*) > n$ and it is then sufficient to ensure that this doesn't happen for any $\eta \in \mct_{[k_*]}$. 
Since the number of pieces in each successor partition is finite and $n$ is finite, this can be expressed by a disjunction of all legal possibilities. 

If we now modify (a) by adding finitely may conditions on the $y_i$'s, of the form $P_\eta(y_i)$ or $\neg P_\eta(y_i)$ for some $\eta \in \mct$,  
and by adding finitely many conditions on $x$ of the form $Q_s(x)$ or $\neg Q_s(x)$ for some $s \in \mcs$, these conditions 
simply affect which of all the possible ``legal possibilities'' remain legal. 

\br

(b) given $Q(y_0), \dots, Q(y_n)$, does there exist $x$ s.t. $P(x) \land \bigwedge_{i\leq n} R(y_i,x)$ ?

Let $k_*$ be minimal so that $f(k_*) > n$.   
Informally, the answer to (b) is yes if and only if for some $\ell \leq f(k_*+1)$, and some $\eta \in \mct_{k_*}$, 
$\bigwedge_{i \leq n} (R(y_i, x) \land P_{\eta^\smallfrown \langle \ell \rangle}(x) )$ is consistent.\footnote{Note this is equivalent to saying: for {every} $k \leq k_*$, for some $\ell \leq f(k +1)$, and some $\eta \in \mct_{k}$, 
we have that $\bigwedge_{i \leq n} (R(y_i, x) \land P_{\eta^\smallfrown \langle \ell \rangle}(x) )$ is consistent.}
(Why? We may verify by induction on $k \geq k_*$ that there is ${\eta_k} \in \mct_k$ such that 
$\bigwedge_{i \leq n} (R(y_i,x) \land P_{\eta_k}(x))$ is consistent. Let $\eta_{k_*} = \eta$. 
At $k > k_*$,  for each $i \leq n$, $R(y_i,x)$ will be inconsistent with precisely one of the formulas  
$\{ P^M_{{\eta_k}^\smallfrown\langle \ell \rangle} : \ell \leq f(k+1) \}$, so $n$ is not large enough to rule out one which works for all. Conversely, 
if the answer to (b) is no, we can bound the height of an inconsistency by noting that each $y_i$ which is consistent with $P_\nu$ can only miss one 
piece of the successor partition $\{ P^M_{\nu^\smallfrown\langle \ell \rangle} : \ell \leq f(\lgn(\nu)) \}$.) 

Since the number of pieces in each successor partition is finite and $n$ is finite, (b) can be expressed by a disjunction over all legal possibilities, using the 
possible types for the $y_i$ as expressed by the $Q_s$'s.  As before, if we modify (b) by adding conditions on the $y_i$'s expressing which 
pieces of which successor partitions they may miss (expressible by the $Q_s$'s) and by adding conditions on $x$ expressing in which pieces of which partitions it may be 
(expressible by the $P_\eta$'s), these conditions 
simply affect which of all the possible ``legal possibilities'' remain legal. 

In both (a) and (b) above, recall \ref{ntn-f} tells us that $k_* \leq n$. 

\br

\noindent\emph{Part II.}
Second, let us extend this analysis to show that
\begin{quotation}
\noindent $(\star)$ if $M$ is a model of $T_{f,k}$, $N$ is a model of $T_{f,k+1}$, 
and $\psi$ is a sentence of $\tau_{k}$ of length $\leq k$, then 
$M \models \psi \iff N \models \psi$. 
\end{quotation}
This will tell us that the sequence $\langle T_{f,k} : k <\omega \rangle$ from \ref{tof} converges to a theory we may call $T_f$, 
and we can conclude that for every 
formula $\vp(\bar{x})$ of $\tau$, for some quantifier free $\psi(\bar{x})$, for every $k<\omega$ large enough, we have 
\[ (\forall \bar{x}) (~ \psi(\bar{x}) \equiv \vp(\bar{x}) ~) \in T_{f}.\]
Fix $k< \omega$.  Let  
$f = \{ (a_\ell, b_\ell) : \ell < k \}$ be a partial one to one map from $M$ into $N$ with $|\dom(f)| \leq k$ such that 
for every atomic, or equivalently, quantifer-free formula $\psi(x_0, \dots, x_{k-1})$ of $\tau_k$, we have that 
$M \models \psi[a_0, \dots, a_{k-1}]$ if and only if $N \models \psi[b_0, \dots, b_{k-1}]$. The point is that $N$ and 
$M$ agree on the information up to level $k$, but $N$ has more information as it is a $\tau_{k+1}$-model. 
Let us show that for every $a \in M$ there is $b \in N$ such that 
$f \cup \{ (a,b) \}$ is a partial one to one map from $M$ into $N$ such that 
for every quantifier-free formula $\psi(x_0, \dots, x_{k})$ of $\tau_{k+1}$, we have that 
$M \models \psi[a_0, \dots, a_{k-1}, a]$ if and only if $N \models \psi[b_0, \dots, b_{k-1}, b]$. 
(The proof that for every $b \in N$, there is $a \in M$ such that the same holds, is similar.)

As equalities and non-connections are easy to handle, the two main cases to consider are positive instances of 
$R$, first when $Q(a)$ and second when $P(a)$.  We may assume $a \notin \{ a_i : i < k \}$ as otherwise it is trivial. 

Suppose that $Q(a)$ and $\{ P(a_i) : i < i_* \leq k \}$ (we can safely ignore the others in this case). 
Since $M$ is a model, $Q_s(a)$ for some $s \in \mcs_k$, and also for each 
$i < i_*$ there is $\eta_i \in \mct_k$ such that $P_{\eta_i}(a_i)$.  Suppose $R(a,a_i)$ for eack $i< i_*$. Then 
in $N$,  $P_{\eta_i}(b_i)$ for each $i<i_*$ by hypothesis, where we know that $\{ \eta_i : i < k \} \subseteq s$ so do not 
include any full splitting. Now in $N$, which is a $\tau_{k+1}$-structure, for each $i < i_*$ there is $\ell_i < f(k+1)$ such that 
\[ P_{{\eta_i}^\smallfrown \langle \ell_i \rangle} (b_i). \]
Then $| \{ \ell_i : i < i_* \} | \leq k$, so recalling (a) earlier in the proof, 
there must be $s^\prime$ such that $s \subseteq s^\prime \in \mcs_k$ and 
$\{ {\eta_i}^\smallfrown \langle \ell_i \rangle : i < k \} \subseteq \mcs^\prime$.  Then we may choose $b \in Q^N_{s^\prime}$ 
which relates to the $b_i$'s as desired.  

Suppose that $P(a)$ and suppose $\{ Q(a_i) : i < i_* \leq k\}$ (we can safely ignore the others in this case). 
Since $M$ is a model, $P_\eta(a)$ for some $\eta \in \mct_k$ and likewise for each $i < i_*$
$Q_{s_i}(a_i)$ for some $s_i \in \mcs_k$. 
Suppose $R(a_i, a)$ for each $i < i_*$. Then in $N$, being a $\tau_{k+1}$-structure, for each $i < i_*$ there 
is $s^\prime_i \in \mcs_{k+1}$ such that $s_i \subseteq s^\prime_i$ and
\[  \{ Q_{s^\prime_i}(b_i) : i < i_* \}. \]
The existence of $a$ tells us that $\eta \in s^\prime_i$ for each $i< i_*$, so as $i_* \leq k$, 
we may find $\ell < f(k+1)$ such that 
${\eta^\smallfrown \langle \ell \rangle} \in s^\prime_i$ for all $i<i_*$
Recalling (b) earlier in the proof, this tells us that 
\[ \bigwedge_{i < i_*}  Q_{s^\prime_i}(y_i) \land R(y_i,x) \land P_{\eta^\smallfrown \langle \ell \rangle}(x) \]
is consistent, and so we may choose $b$ in  $P^N_{\eta^\smallfrown \langle \ell \rangle}$ to complete the proof. 
\end{proof}

\br

Before proceeding to an analysis of saturation and non-saturation in ultrapowers of models of $\tsf$, let us briefly describe the types we will be dealing with. 
Recall that $\trg$ is the theory of the random graph. 

\begin{claim} \label{up-types}
Let $M \models \tsf$. Let $I$ be any infinite set, $|I| = \lambda$, and $\de$ a regular ultrafilter on $I$ which is good for $\trg$. Then to show $N = M^I/\de$ is $\lambda^+$-saturated, 
it suffices to show that $N$ realizes all partial types of the following form:

\begin{enumerate}
\item[(a)] $\{ Q(x) \} \cup \{ R(x,a) : a \in A \}$ for $A \subseteq P^N$, $|A| \leq \lambda$. 
\item[(b)] $\{ P(x) \} \cup \{ R(b,x) : b \in B \}$ for $B \subseteq Q^N$, $|B| \leq \lambda$. 
\end{enumerate}
\end{claim} 

\begin{proof}
Fix a model $N$. 
To analyze saturation of $N$, it suffices to consider $1$-types, so there are two cases: the type contains $P(x)$, or it contains $Q(x)$. 
Since the ultrafilter is good for the random graph, necessarily $\mu(\de) \geq \lambda^+$ (i.e. any pseudofinite set has size at least $\lambda^+$).
Recall that for models with countable vocabularies, which is always our case here, saturation of ultrapowers reduces to saturation of $\vp$-types 
(\cite{mm1}, Theorem 12). [So we may hope to not need to consider types with infinitely many distinct $Q_s$'s or $P_\eta$'s.]

Let $p \in \ts(C, N)$, $C \subseteq N$, $|C| \leq \lambda$, $p$ not algebraic. 

Without loss of generality, 
\begin{enumerate}
\item[(i)] $N \rstr C \preceq N$.
\item[(ii)] if $\eta \in \mct$, $Q(x) \in p(x)$ and $(\exists z) (P_\eta(z) \land R(x,z)) \in p$, then for some $a_\eta \in P^N_\eta$ we have 
$R(x,a_\eta) \in p$. 
\item[(iii)] if $s \in \mcs$, $P(x) \in p(x)$ and $(\exists z) (Q_s(z) \land R(z,x)) \in p$, then for some $b_s \in Q^N_s$ we have 
$R(b_s, x) \in p$. 
\end{enumerate}
[We can assume this because for all $\eta \in \mct$ and all $s \in \mcs$, $|P^N_\eta|$, $|Q^N_s|$ are $> \lambda$ because $\de$ is regular 
and $N = M^I/\de$.] 

Now let: 
\begin{enumerate}
\item[(a)] $A_1 = \{ a \in C : P(x) \land R(a,x) \in p $ or $ Q(x) \land R(x,a) \in p \}$. 
\item[(b)] $A_0 = \{ a \in C : P(x) \land \neg R(a,x) \in p $ or $ Q(x) \land \neg R(x,a) \in p \}$. 
\item[(c)] if $P(x) \in p(x)$ then let:
\begin{itemize}
\item $\eta = \eta_p \in \lim(\mct)$ be such that $p_2(x) := \{ P_{\eta \rstr \ell}(x) : \ell < \omega \} \subseteq p(x)$, 
\item $p_1(x) = \{ R(a,x) : a \in A_1 \}$, 
\item $p_0(x) = \{ \neg R(a,x) : a \in A_0 \}$, and 
\item $p_3(x) = \{ x \neq c : c \in C $ and $c$ realizes 
$\{ P_{\eta_p \rstr n}(x) : n < \omega \} \}$.
\end{itemize}
\item[(d)] if $Q(x) \in p(x)$ then let: 
\begin{itemize}
\item $\nu = \nu_p \in \lim(\mcs)$ be such that $p_2(x) := \{ Q_{\nu(\ell)}(x) : \ell < \omega \} \subseteq p(x)$, 
\item $p_1(x) = \{ R(x,a) : a \in A_1 \}$, and 
\item $p_0(x) = \{ \neg R(x,a) : a \in A_0 \}$. 
\item $p_3(x) = \{ x \neq c : c \in C $ and $c$ realizes $\{ Q_{\nu_p \rstr n }(x) : n < \omega \} \}$.
\end{itemize}
\end{enumerate}
The $p$ is equivalent to $p_1(x) \cup p_0(x) \cup p_2(x) \cup p_3(x)$, by elimination of quantifiers and our assumptions (i),(ii),(iii). 

Now if $\eta \in \mct$, $P_\eta(x) \in p$ then $p_1(x) \vdash P_\eta(x)$. 
[Why? Let $k = \lgn(\eta)$ and let $\mcs^* = \{ s \in \mcs_k : \eta \in s \} \subseteq \mcs_k$. So for each $s \in \mcs^*$, there is 
$b_s \in Q^N_s$ such that $R(b_s,x) \in p$ by (ii), so it is easy to see that 
$\{ R(b_s,x) : s \in \mcs^* \} \vdash P_\eta(x)$.]

Similarly, if $Q_s(x) \in p(x)$ for $s \in \mcs$ then $p_1(x) \vdash Q_s(x)$. 
So in both cases,
\[ p_1(x) \vdash p_2(x). \]
It remains to handle $p_0 \cup p_3$. Suppose, then, that our ultrapower realizes $p_1(x)$.  
The sets $A_1, A_0 \subseteq N$ defined in (a), (b) are disjoint and both of size $\leq \lambda$. 
As $\de$ is regular, there are disjoint pseudo-finite internal sets $X_0$, $X_1$, $X_2$ such that $A_\ell \subseteq X_\ell$
for $\ell = 0, 1$ and $X_0 \cap X_1 = \emptyset$ and $X_2$ disjoint to $X_0 \cup X_1$ (even, $\subseteq P$ if 
$X_0 \cup X_1 \subseteq Q$, etc). Suppose $c \in N$ realizes $p_1(x)$. There are two 
cases depending on whether $Q(x) \in p$ or $P(x) \in p$. If $Q(x) \in p$, there is $c^\prime \in N$ such that:

\begin{quotation}
\begin{itemize}
\item $(\forall y \in X_1) (R(c^\prime,y) \iff R(c, y))$
\item $(\forall y \in X_0)(\neg R(c^\prime, y))$.
\end{itemize}
\end{quotation}
So $c^\prime$ realizes $p_1(x)$, $p_2(x)$, $p_0(x)$ and even $p_3(x)$, so we are done.  If $P(x) \in p$, replace $R(c^\prime,y)$ by $R(y,c^\prime)$ and 
$R(c,y)$ by $R(y,c)$ in the above quotation.  This completes the proof. 
\end{proof}

\vspace{5mm}

\section{A non-saturation result for $\tsf$} \label{t:ns}

\begin{ntn}
When $\langle a_\alpha : \alpha < \kappa \rangle$, and $w \subseteq \kappa$, write $\bar{a}_w$ to mean $\langle a_\alpha : \alpha \in w \rangle$. 
\end{ntn}

Recall also our notation for $\mct$, $\mct_{[k]}$, etc. from \ref{orig-ntn}. 

\begin{lemma} \label{c8}
Let $f \in \mcf$. 
Let $\lambda$ be any infinite cardinal and $\mu < \kappa$.  
Then no ultrafilter $\de$ of $\ba = \ba^1_{2^\lambda, \mu, \aleph_0}$ is $(\kappa^+, \tsf)$-moral. 
\end{lemma}

\begin{proof}
\setcounter{equation}{0}
We use the setup of separation of variables, so we have in mind the background set $I = \lambda$, a homomorphism 
$\jj: \mcp(\lambda) \rightarrow \ba$, an ultrafilter $\de_*$ on $\ba$, and also a model $M \models \tsf$ and an ultrapower $M^I/\de$. 
We will build a possibility pattern in the variables 
$x$ and $\langle x_\alpha : \alpha  < \kappa \rangle$ and show it doesn't have a multiplicative 
refinement.\footnote{The proof will take place entirely in $\ba$, but it is simply asserting that we could build a type $\{ R(x,a_\alpha) : \alpha < \kappa \}$ in the ultrapower $M^I/\de$, i.e. we 
could choose our parameters $a_\alpha$, in such a way that each $\ma[\vp(\bar{x}_w)]$ in the proof below is really the image of 
$\{ t \in I : M \models \vp[\bar{a}_w[t]] \}$ under $\jj$, and likewise for each finite $u \subseteq \kappa$, $\mb_u$ is the image of 
$\{ t \in I : M \models \exists x \bigwedge_{\alpha \in u } R(x,a_\alpha[t]) \}$ under $\jj$.}

Recall that $\ba^+ = \ba \setminus \{ 0 \}$. 
\begin{equation} \label{ind-constr}
\mbox{For each $\rho \in \mct_k$ and $\alpha < \kappa$,  
define $\ma[P_{\rho}(x_\alpha)] \in \ba^+$ by induction on $k<\omega$:}
\end{equation}
\begin{itemize}
\item[(a)] if $k= 0$, i.e. $\rho \rstr k = \langle \rangle$,   $\ma[P_{\langle \rangle}(x_\alpha)] = 1_\ba$.
\item[(b)] for $k$ and $i \leq f(k)$, 
let $g_{\alpha, k, i}$ be the function with domain $\{ \omega \alpha +  k \}$ such that  
$g_{\alpha, k, i} (\omega \alpha +  k) = i$.  
Then for $\rho \in \mct_k$, $\ell \leq f(k)$ define   
\[ \ma[P_{\rho ~^\smallfrown \langle \ell \rangle}(x_\alpha)] = \ma[P_{\rho} (x_\alpha)] \cap \mx_{g_{\alpha, k, \ell}}. \]
\item[(c)] for $\alpha \neq \beta < \kappa$, 
\[ \ma[x_\alpha \neq x_\beta] = 1_\ba. \]
\end{itemize}
Without loss of generality we assume 
\begin{equation} \label{star-2}
\mx_{g_{\alpha, k, 0}} \in \de_* \mbox{ for } \alpha < \kappa.
\end{equation}
By (\ref{ind-constr})(c), 
\begin{equation} \label{star-3}
\ma[x_\alpha \neq x_\beta] \in \de_*
\end{equation} 
and by quantifier elimination, 
$\ma[\vp(\bar{x}_v)]$ is determined for any finite $v \subseteq \kappa$ and any $\vp(\bar{x}_{v})$ in the language of $T$.  
It follows that 
for each $\alpha < \kappa$ and each $k < \omega$, $\ma[P_{\langle 0_k \rangle}(x_\alpha)] \in \de_*$, where $\langle 0_k \rangle$ is $\langle 0, \dots, 0 \rangle$ ($k$ times). 
By (\ref{ind-constr}), (\ref{star-2}), (\ref{star-3}), 
the sequence 
\begin{equation}
\bar{\mb} = \langle \mb_u : u \in [\kappa]^{<\aleph_0} \rangle~ \mbox{ where } \mb_u = \ma[\exists x \bigwedge_{\alpha \in u } R(x, x_\alpha)] 
\end{equation}
is a possibility pattern. [That is, in the related ultrapower we are considering the type 
$\{ R(x,a_\alpha) : \alpha < \kappa \}$ where the $a_\alpha$'s are pairwise distinct realizing the type 
$\{ P_{\rho_*\rstr k}(x) : k < \omega \}$ where $\rho_* = (0,0, \dots)$.\footnote{The intent is that the elements $a_\alpha$ \underline{all} satisfy the 
``constant zero'' branch in the ultrapower, so $p(x)$ is clearly a type. The potential problem in realizing the type is that in each index model, 
by our construction (\ref{ind-constr}), the elements $a_\alpha[t]$ may project to different branches ``$\de$-rarely'', possibly violating the condition on splitting.}]

Towards contradiction assume 
\begin{equation} 
\label{b-prime}
\bar{\mb}^\prime = \langle \mb^\prime_u : u \in [\kappa]^{<\aleph_0} \rangle
\end{equation} is a multiplicative refinement of $\bar{\mb}$. 

For $\alpha < \kappa$ choose $g_\alpha \in \fin_{\aleph_0}(2^\lambda)$ such that $\mx_{g_\alpha} \leq \mb^\prime_{\{\alpha\}}$. 
We will repeatedly use the fact that no set of size $>\mu$ (e.g. no set of size $\kappa$) is covered by a countable union of sets of cardinality $\leq \mu$. 
Given an ordinal $\beta$, let us say ``the remainder of $\beta$ mod $\omega$ is $k$'' to mean that $\beta = \omega \gamma + k$ for some ordinal $\gamma$ and integer $k$. 

Since each $g_\alpha$ has finite domain, there is a smallest integer $k_\alpha$ such that for every $\beta \in \dom(g_\alpha)$, the remainder of 
$\beta$ mod $\omega$ is $\leq k_\alpha$. 
As $\mct_*$ has cardinality $>\mu$, there is some finite $k_*$ such that $\uu_1 = \{ \alpha < \kappa : k_\alpha = k_* \}$ is of cardinality $>\mu$. 

For each $\alpha \in \uu_1$, the elements 
\[   \{ \ma[P_{\nu}(x_\alpha)] : \nu \in \mct_{k_*} \} \]
form a maximal antichain of $\mb_\alpha$. 
For each $\alpha \in \uu_1$, choose $\nu_\alpha \in \mct_{k_*}$ such that 
\[ \mx_{g_\alpha} \cap \ma[P_{\nu_\alpha}(x_\alpha)] > 0. \]
By the construction in (\ref{ind-constr}) above, we can translate as follows: 
letting 
\[ g^*_\alpha = g_\alpha \cup \bigcup_{k \leq k_*} \{ (\omega \alpha + k, \nu_\alpha(k) ) \} \]
we have that 
\begin{equation}
\label{eq51} 
\mx_{g^*_\alpha} = \mx_{g_\alpha} \cap \ma[P_{\nu_\alpha}(x_\alpha)] > 0. 
\end{equation}
Moreover, for each $\alpha \in \uu_1$, it is still the case that 
for every $\beta \in \dom(g^*_\alpha)$, the remainder of $\beta$ mod $\omega$ is still $\leq k_*$. 
Since $\mct_{k_*}$ is finite, for some $\nu_* \in \mct_{k_*}$, 
\[ \uu_2 = \{  \alpha \in \uu_1 : \nu_\alpha = \nu_* \} \]
is of cardinality $>\mu$. 

Choose some very large finite $m$ (at least $f(k_*+1)$ suffices, see below). Apply Claim \ref{intersections} to $\langle g^*_\alpha : \alpha \in \uu_2 \rangle$ and $m$, so the Claim says: \emph{Suppose $\langle g^*_\alpha : \alpha \in \uu_2 \rangle$ 
is a subset of $\fin_{\mu,\aleph_0}(2^\lambda)$, where $\uu_2$ is of cardinality $>\mu$, and 
$m <\omega$.  Then there is $u \subseteq \uu_2$, $|u| = m$ such that $\bigcup_{\alpha \in u} g^*_\alpha$ is a function.} 

Let $u \subseteq \uu_2$ be as returned by Claim \ref{intersections}. Then $g^* = \bigcup_{\alpha \in u} g^*_\alpha$ is a function, and 
it is still the case that for all $\beta \in \dom(g^*)$, the remainder of $\beta$ mod $\omega$ is $\leq k_*$. 
Since we assumed (for a contradiction) that $\bar{\mb}^\prime$ is multiplicative and refines $\bar{\mb}$, and each $\mx_{g_\alpha} \leq \mb^\prime_\alpha$, 
we have that for every $v \subseteq u$, 
\begin{equation}
\label{e52}
0 <  \mx_{g^*} \leq \bigcap_{\alpha \in v} \mx_{g_\alpha}  \leq \bigcap_{\alpha \in v} \mb^\prime_\alpha = \mb^\prime_v \leq \mb_v. 
\end{equation}
Thus, for every $\ell \leq f(k_*)$ and every $\alpha \in u$,  
\[ \dom(g^*) \cap   \{ \omega \alpha + k_* + 1\} = \emptyset \]
and clearly for $\alpha \neq \alpha^\prime \in u$, 
\[ \{ \omega \alpha + k_* + 1 \} \cap \{ \omega \alpha^\prime + k_* + 1 \} = \emptyset. \]
So for any $f(k_*+1)$ distinct $\alpha$'s in $u$, say we enumerate them as $\langle \alpha_{i_\ell} : \ell \leq f(k_* +1) \rangle$, 
\[ \mx_{g_*} \cap  \bigcap_{\ell \leq f(k_*+1)} \ma[P_{\nu_*~ ^\smallfrown \langle \ell \rangle}(x_{\alpha_{i_\ell}})] > 0 \]
in other words, recalling (\ref{eq51}), 
\[ \bigcap_{\ell \leq f(k_*)} \mx_{g_{\alpha_\ell}} \cap \bigcap_{\ell \leq f(k_*)} \ma[P_{\nu^*}(x_{\alpha_{i_\ell}})] \cap  \bigcap_{\ell \leq f(k_*+1)} \ma[P_{\nu_*~ ^\smallfrown \langle \ell \rangle}(x_{\alpha_{i_\ell}})] > 0 \]
Rewriting $v = \{ \alpha_{i_\ell} : \ell \leq f(k_*) \} \subseteq u$, this contradicts the fact from (\ref{e52}) that 
$\bigcap_{\alpha \in v} \mx_{g^*_\alpha} \subseteq \bigcap_{\alpha \in v} \mx_{g_\alpha} \leq \mb_v$. 
So there can be no such $\bar{\mb}^\prime$, which completes the proof.  
\end{proof}

\vspace{5mm}

\section{$\tsf$ is explicitly simple}

This section proves Theorem \ref{tnes}, but its aim is equally or even primarily pedagogical, to exposit a way of measuring simple theories.  
In \cite{MiSh:1030} we defined ``explicit simplicity,'' a way of stratifying the complexity of simple theories using cardinals 
$(\lambda, \mu, \theta, \sigma)$ satisfying Definition \ref{d:suitable} below. 
This is motivated in the introduction to \cite{MiSh:1030}, \S 3. 
The definition of ``$(\lambda, \mu, \theta, \sigma)$-explicitly simple'' implies that 
$T$ is simple, and it follows from the definitions that this becomes weaker as $\mu$ increases. \cite{MiSh:1030} Theorem 4.10 had shown that this definition holds for any simple theory $T$ 
when $|T| < \sigma$ and we use the largest nontrivial number of ``colors,'' $\mu^+=\lambda$. 

This new characterization of 
simplicity suggested a program of stratifying simple theories according to the necessary value of $\mu$. For the random graph one color is enough, but it 
turned out in \cite{MiSh:1050} that the case, of, say, the tetrahedron-free three-hypergraph, 
needs either $\mu^+ = \lambda$ or $\mu^{++}=\lambda$.  
Moreover, the idea of the infinite descending chain of \cite{MiSh:1050} is essentially to look for theories whose $\mu$ must satisfy $\mu < \lambda \leq \mu^{+n}$ 
for larger and larger finite $n$,\footnote{This is motivation, not a mathematical statement: \cite{MiSh:1050} isn't carried out in the language of explicit simplicity, 
since there we had a concrete family of theories with trivial forking and so could simplify 
the definitions somewhat, e.g. dropping the requirement that closures be submodels.} inspired by ideas on free sets in set mappings and the 
Kuratowski-Sierpinski characterization of the $\aleph_n$'s (for more on set mappings, see \cite{ehmr}, \cite{KoSh:645}). 

We left as an open question, 
Question 10.1 of \cite{MiSh:1030}, whether it was possible to build a simple theory where $\mu$, the range of the coloring function, 
is truly uncountable but does not depend on $\lambda$. 

In this section we show that $\tsf$ is \emph{not} such an example. In light of \S \ref{t:ns}, Theorem \ref{tnes} tells us that when $\sigma > |T|$, uncountable $\mu$ is not necessary to be in a class strictly above the random graph. 
This is a delicate point: it highlights that 
the definition of explicit simplicity \emph{requires} $\sigma > |T|$, with the closure of a finite set 
(in the relevant algebra) giving rise to an elementary submodel.   So $\sigma > |T|$ is an assumption in Theorem \ref{tnes} even though 
$\theta > |T|$ would seem more natural for our case. 
For, as we will see, $\sigma > |T|$ and  knowing a type of $\tsf$ over an elementary submodel is already enough 
to control consistency of its automorphic images.  Indeed, if $\sigma$ were finite, the story would be different\footnote{It is natural to give a general definition of explicit simplicity for $\sigma = \aleph_0$ which 
allows the closure of finite sets to be finite, and applies to theories with trivial forking, see after the proof.}: we will see at the end of the proof below a suggested 
strategy for saturation of ultrapowers in \S \ref{s:sat}. 

We now review the setup.  

\begin{defn}[\cite{MiSh:1030} Definition 1.1] \label{d:suitable}
Call $(\lambda, \mu, \theta, \sigma)$ \emph{suitable} when:
\begin{enumerate}
\item $\sigma \leq \theta \leq \mu < \lambda$
\item $\theta$ is regular, $\mu = \mu^{<\theta}$ and $\lambda = \lambda^{<\theta}$
\item $(\forall \alpha < \theta)(2^{|\alpha|} < \mu)$.
\end{enumerate}
\end{defn}

For Theorem \ref{tnes}, the needed definitions will be given in the course of the proof. 

\begin{theorem}\label{tnes} The theory $\tsf$ is $(\lambda, \mu, \theta, \sigma)$-explicitly simple whenever $(\lambda, \mu, \theta, \sigma)$ 
is suitable and $\sigma > |\tsf|$. 
\end{theorem}

\begin{proof} 
We will simply follow \cite{MiSh:1030}, Section 3, but all relevant definitions have been quoted for ease of reading. 
We are aiming for Definition \ref{d:c3}; all the terms mentioned in the definition will be defined subsequently. 

\begin{quotation}
\begin{small}
\noindent
\begin{defn}[Explicitly simple, \cite{MiSh:1030} Definition 3.2] \label{d:c3} 
Assume $(\lambda, \mu, \theta, \sigma)$ are suitable. We say 
$T$ is $(\lambda, \mu, \theta, \sigma)$-\emph{explicitly simple} if $T$ is simple and 
for every $N \models T$, $||N|| = \lambda$, $p \in \ts(N)$ nonalgebraic, 
\begin{enumerate}
\item[(a)] there exists a $(\lambda, \theta, \sigma)$-presentation $\xm$ of $p$.

\item[(b)] for every $(\lambda, \theta, \sigma)$-presentation $\xm$ of $p$, there is a presentation $\xn$ of $p$ refining $\xm$ and a function $G: \mcr_{\xn} \rightarrow \mu$ 
such that $G$ is an intrinsic coloring of $\mcr_\xn$. 
\end{enumerate}
\end{defn}
\end{small}
\end{quotation}
$\tsf$ is indeed simple. So suppose we are given  $N \models T$, $||N|| = \lambda$, $p \in \ts(N)$ nonalgebraic.  
{We will have two cases: the case where $P(x) \in p$, and the case where $Q(x) \in p$. The only difference will come at the end.}

\begin{quotation}
\begin{small}
\noindent
\begin{defn}[Presentation, \cite{MiSh:1030} Definition 3.3] \label{d:pres}
Suppose we are given $N \models T$, $||N|| = \lambda$, and $p \in \ts(N)$. 
A \emph{$(\lambda, \theta, \sigma)$-presentation} for $p$ is the data of an enumeration and an algebra, 
\[ \xm = (\langle \vp_\alpha(x, a^*_\alpha) : \alpha < \lambda \rangle, \zm ) \]
where these objects satisfy: 

\begin{enumerate}
\item $p = \langle \vp_\alpha(x;a^*_\alpha) : \alpha < \lambda \rangle$ is an enumeration of $p$, 
which induces an enumeration $\langle a^*_\alpha : \alpha < \lambda \rangle$ of $\dom(N)$, possibly with repetitions, and with the $a^*_\alpha$ 
possibly imaginary. 

\item $\zm$ is an algebra on $\lambda$ with $<\theta$ functions. 

\item For any finite $u \subseteq \lambda$, $|\clm(u)| < \sigma$. Thus, for any $u \subseteq \lambda$, 
if $|u| < \sigma$ then $|\clm(u)| < \sigma$, and if $|u| < \theta$ then $|\clm(u)| < \theta$.

\item $\clm(\emptyset)$ is  
an infinite cardinal $\leq |T|$, so an initial segment of $\lambda$. 
\\ $M_* := N \rstr \{ a^*_\alpha : \alpha < \clm(\emptyset) \}$ is a distinguished elementary submodel of $N$, 
and we require that $p$ does not fork over $M_*$. 

\item Moreover, for each $u \in [\lambda]^{<\sigma}$, $N_u := N \rstr \{ a^*_\alpha : \alpha \in \clm(u) \}$ is an elementary submodel of $N$, 
and $\{ \vp_\alpha(x,a^*_\alpha) : \alpha \in \clm(u) \}$ is a complete type over this submodel which dnf over $M_*$. $($In particular, 
$\{ \vp_\alpha(x, a^*_\alpha) : \alpha \in \clm(\emptyset) \}$ is a complete type over $M_*$.$)$

\item If $\alpha \in \clm(u)$, $\beta \leq \alpha$, writing $A_\beta = \{ a^*_\gamma : \gamma < \beta \}$, we have that  
\\ $\tp(a^*_\alpha, A_\beta \cup M_*)$ does not fork over $\{a^*_\gamma : \gamma \in \clm(u) \cap \beta \} \cup M_*$.

\end{enumerate}
In a context where $(\lambda, \theta, \sigma)$ are given, ``presentation'' means ``$(\lambda, \theta, \sigma)$-presentation.''
\end{defn}
\end{small}
\end{quotation}
To show $p$ has a presentation, first fix a countable elementary submodel $M_* \preceq N$. Choose an enumeration 
$\langle \vp_\alpha(x,a^*_\alpha) : \alpha < \lambda \rangle$ of $p$ so that:
\begin{itemize}
\item $\{ a^*_\alpha: \alpha < \omega \} = \dom(M_*)$,
\item $\{ \vp(x,a^*_\alpha) : \alpha < \omega \} = p \rstr M_*$, 
\item and $\{ a^*_\alpha : \alpha < \lambda \} = \dom(N)$, 
\end{itemize}
noting the sequence $\langle a^*_\alpha : \alpha < \lambda \rangle$ may have repetitions. 
This is easily done as $x \neq a$ belongs to $p$ for every $a \in N$. 
For the algebra $\zm$, we add three kinds of functions. 
\begin{itemize}
\item First choose countably many unary functions $\{ g_n : n < \omega \}$, where $g_n$ is the constant function $n$, to 
ensure that the ``closure of the empty set''\footnote{we never consider $\emptyset$ as a base set, so this effectively is the 
set which is contained in every closure of every nontrivial set; or if you prefer, consider an algebra to be a structure on $\lambda$ with functions and no relations 
\emph{and a constant} (interpreted as any element of $\omega$).} is $\omega$. 
\item Second, choose countably many 
functions which are analogues of Skolem functions for $\tsf$. That is, for each formula $\vp(x,\bar{y})$ of $\ml(\tau(\tsf))$, 
let $f_{\vp(\bar{y})}$ be a new function symbol, and interpret these countably many new function symbols as Skolem functions for 
$\tsf$ in $N$, in each case choosing the witness $a^*_\alpha$ of smallest index $\alpha < \lambda$. 
Then for each $\vp(x,\bar{y})$ add to the algebra a new $\ell(\bar{y})$-place function symbol 
$g_{f_{\vp(x,\bar{y})}}$ which mirrors the action of the Skolem function on the indices $\lambda$:  $g_{f_{\vp}}(\bar{u}) = v$ only if 
$f_{\vp}(\bar{a}^*_u) = a^*_v$. 
\item Third, we want to ensure that the type restricted to closed sets is complete. For each $\psi(x,y) \in \{ R(x,y), x=y, Q(x) \land y=y, P(y) \} \cup \{ Q_s(x) : s \in \mcs \} \cup 
\{ P_\eta(y) : \eta \in \mct \}$, let $h_\psi(y)$ be defined so that $h_\psi(\alpha) = \beta$ if $\beta < \lambda$ is the least ordinal such that 
$\vp_\beta(x,a^*_\beta)$ is equivalent mod $T$ to either $\psi(x,a^*_\alpha)$ or its negation.\footnote{It's sufficient if the restriction of $p$ to a closed set generates a complete type. 
The reason to ask that $\{ \vp(x,a^*_\alpha) : \alpha < \omega \} = p \rstr M_*$ above with equality instead of $\vdash$ was just to ensure the closure of the empty set didn't grow.}
\end{itemize}
Since there is no nontrivial forking, this suffices, and 
$( \langle \vp_\alpha(x, a^*_\alpha) : \alpha < \lambda \rangle, \zm )$ is a presentation. 
\begin{quotation}
\begin{small}
\noindent
\begin{defn}[Refinements of presentations, \cite{MiSh:1030} Definition 3.6] \label{d:extend} 
Suppose we are given $N \models T$, $||N|| = \lambda$, and $p \in \ts(N)$. Let $\xm = (\bar{\vp}_\xm, \zm_\xm)$, 
$\xn = (\bar{\vp}_\xn, \zm_\xn)$ be presentations of $p$. 
We say that \emph{$\xn$ refines $\xm$} when: 
\begin{enumerate}
\item[(a)] $\bar{\vp}_\xm = \bar{\vp}_\xn$. 
\item[(b)] $\operatorname{cl}_{\zm_\xm}(\emptyset) = \operatorname{cl}_{\zm_\xn}(\emptyset)$. 
\item[(c)] $\zm_\xm \subseteq \zm_{\xn}$, i.e. the algebra of $\xn$ extends that of $\xm$.  
\end{enumerate}
\end{defn}
\end{small}
\end{quotation}
In a refinement, the enumeration stays the same, the distinguished elementary submodel stays the same, but we may add a few more functions to the algebra if we wish. 
In our case it isn't necessary; we'll just show directly that every presentation has a coloring.  For the rest of the proof, then, assume we have been given some 
fixed presentation $\xm$. 

\begin{quotation}
\begin{small}
\noindent
\begin{defn} \label{d:es3} \emph{(The set of quadruples $\mcr_{\xm}$, \cite{MiSh:1030} Definition 3.9)}
Let $\xm$ be a presentation of a given type $p = p_{\xm}$.  
Then $\mcr = \mcr_{\xm}$ is the set of $\xr = (u, w, q, r)$ such that:
\begin{enumerate}
\item $u \in [\lambda]^{<\sigma}$, $w \in [\lambda]^{<\theta}$ and $w = \clm(w)$.
\item $u \subseteq \clm(u) \subseteq w$. 
\item $q = q(\overline{x}_w)$ is a complete type in the variables $\overline{x}_{w}$ such that:
\begin{enumerate}
\item for any finite $v \subseteq \clm(\emptyset)$, if $M_* \models \psi(\overline{a}^*_v)$ then $\psi(\overline{x}_v) \in q$.
\item for any finite $\{ \alpha_0, \dots, \alpha_n \} \subseteq u$, $\exists x \bigwedge_{i\leq n} \vp_\alpha(x,a^*_\alpha) ~\in q$.
\end{enumerate} 
\item $r = r(x,\overline{x}_w)$ is a {complete type} in the variables $x, \overline{x}_{w}$, 
extending 
\[  q(\overline{x}_w) \cup \{ \vp_\alpha(x,x_\alpha) : \alpha \in u \}. \] 
\item \label{here} \underline{if} $\overline{b}^*_{w}$ realizes $q(\overline{x}_{w})$ in $\mathfrak{C}_T$ and $\alpha < \clm(\emptyset) \implies b^*_\alpha = a^*_\alpha$,
\underline{then} 
\begin{enumerate}
\item $r(x,\overline{b}^*_{w})$ is a type which does not fork over $M_*$ and extends $p\rstr M_*$.
\item if $w^\prime \subseteq w$ is $\mlx$-closed, $\mathfrak{C}_T \rstr \{ b^*_\alpha : \alpha \in w^\prime \} \preceq \mathfrak{C}_T$ and 
$r(x,\overline{b}^*_{w}) \rstr \overline{b}^*_{w^\prime}$ is a complete type over this elementary submodel.
\item if $w^\prime \subseteq w$ is $\mlx$-closed and $\alpha \in w^\prime$ then $tp(b^*_\alpha, \{ b^*_\beta : \beta \in w \cap \alpha\})$ 
dnf over $\{ b^*_\beta : \beta \in w^\prime \cap \alpha \}$. 
\end{enumerate}
\end{enumerate}
\end{defn}
\end{small}
\end{quotation}
Note that $u$ need not be closed. So in our case, $r$ will describe a type in the variables $x, \bar{x}_w$ which agrees with $p \rstr M_*$ on $\{ x_\alpha : \alpha < \omega \}$;  
it will then contain new conditions stating that $x$ connects to additional elements $x_\alpha$ and stating in which ``leaves'' of the tree these $x_\alpha$'s fall. 
[For example, if $\{ \vp_\alpha(x,x_\alpha) : \alpha \in u \} = \{ R(x,x_\alpha) : \alpha \in u \}$, 
a priori $x_\alpha$ need not be in the same leaf as $a^*_\alpha$ for $\alpha \geq \omega$.]  

\begin{quotation}
\begin{small}
\noindent
\begin{defn}[A non-triviality condition, \cite{MiSh:1030} Definition 3.10] \label{d:good-inst}
Suppose we are given $\overline{\xr} = \langle \xr_t = (u_t, w_t, q_t, r_t) : t < t_* < \sigma \rangle$ from $\mcr_{\xm}$. 
Say that $\overline{b}^* = \langle b^*_\alpha : \alpha \in \bigcup_{t} w_{t} \rangle$, with each $b^*_\alpha \in \mathfrak{C}$ $($possibly imaginary$)$, 
is a \emph{good instantiation} of $\overline{\xr}$ when the following conditions hold. 
\begin{enumerate}
\item $\alpha \in \clm(\emptyset) \implies b^*_\alpha = a^*_\alpha$. 
\item for each $t < t_*$, $\overline{b}^*\rstr_{w_{t}}$ realizes $q_t(\overline{x}_{w_{t}})$.
\item for each $t < t^\prime < t_*$, if $v \subseteq w_t \cap w_{t^\prime}$ is finite, then:
\begin{enumerate}
\item for each formula $\psi(\overline{x}_v)$,  
$\psi(\overline{b}^*_v) \in q_t ~\iff~ \psi(\overline{b}^*_v) \in q_{t^\prime}$.
\item for each formula $\psi(x,\overline{x}_v)$,  
$\psi(x,\overline{b}^*_v) \in r_t ~\iff~ \psi(x,\overline{b}^*_v) \in r_{t^\prime}$. 
\end{enumerate}
\item if $\beta \in w_t$ for some $t < t_*$ then
\[ \tp(b^*_\beta, \{ b^*_\gamma : \gamma \in \bigcup_{ s\leq t} w_s ~\mbox{and}~ \gamma < \beta \} ) ~\mbox{dnf over}~
\{ b^*_\gamma : \gamma \in w_t \cap \beta \}. \] 
\item for each $t < t_*$, if $w^\prime \subseteq w$ and $\clm(w^\prime) = w^\prime$ then 
$\mathfrak{C}_T \rstr \{ b^*_\alpha : \alpha \in w^\prime \} \preceq \mathfrak{C}_T$ and 
$r_t(x,\overline{b}^*_{w^\prime})$ is a complete type over this elementary submodel which does not fork over $M_*$ 
$($noting that the domain of $M_*$ is $\{ b^*_\alpha : \alpha \in \clm(\emptyset) \}$ by the first item$)$. 
\end{enumerate}
\end{defn}
\end{small}
\end{quotation}
Definition \ref{d:good-inst} is simply to make the definition of coloring meaningful by ruling out trivial inconsistency, as will be clear from the next definition.

\begin{quotation}
\begin{small}
\noindent
\begin{defn}[Coloring, \cite{MiSh:1030} Definition 3.11] \label{d:es4}
Let $\xm$ be a $(\lambda, \theta, \sigma)$-presentation and $\mcr = \mcr_{\xm}$ be from $\ref{d:es3}$. 
Call $G : \mcr_{\xm} \rightarrow \mu$ \emph{an intrinsic coloring of $\mcr_{\xm}$} if: 
whenever 
\[ \overline{\xr} = \langle \xr_t = (u_t, w_t, q_t, r_t) : t < t_* < \sigma \rangle \]
is a sequence of elements of $\mcr_{\xm}$ and 
$\overline{b}^* = \langle b^*_\alpha : \alpha \in \bigcup_{t<t_*} w_t \rangle$ is a good instantiation of $\overline{\xr}$, 

\emph{if} $G \rstr \{ \xr_t : t < t_* \}$ is constant, 

\emph{then} the set of formulas 
\[ \{ \vp_\alpha(x, b^*_\alpha) \colon ~\alpha \in u_t, ~\vp_\alpha \in r_t, ~ t < t_*  \}  \]
is a consistent partial type which does not fork over $M_*$. 
\end{defn}
\end{small}
\begin{center}
\begin{tiny}END OF QUOTATIONS\end{tiny}
\end{center}
\end{quotation}

\br
\noindent It remains to find a coloring given our fixed $\xm$, and therefore $\mcr$. 

\br

\noindent 
\emph{Case 1: $Q(x) \in p$.}   The surprise in this case is that since $p \rstr M_*$ is determined, we know whether $Q_s(x)$ for all $s \in \mcs_k$ and all $k<\omega$. 
This means the set  $\Lambda$ of possible leaves $\eta$ such that $p \rstr M_* \cup \{ R(x,a) \} \cup \{ P_{\eta \rstr k}(a) : k <\omega \}$ is fixed by $p \rstr M_*$ and inherited 
by any $r$ from some $\xr \in \mcr$.  So \emph{whenever} 
\[ \overline{\xr} = \langle \xr_t = (u_t, w_t, q_t, r_t) : t < t_* < \sigma \rangle \]
is a sequence of elements of $\mcr_{\xm}$ and 
$\overline{b}^* = \langle b^*_\alpha : \alpha \in \bigcup_{t<t_*} w_t \rangle$ is a good instantiation of $\overline{\xr}$, 
it must be the case that for each $t < t_*$ and each $\alpha \in w_t$, $r_t$ determines that the leaf of $b^*_\alpha$ must be $\eta$ for some $\eta \in \Lambda$. 
It follows that 
\[ \{ \vp_\alpha(x, b^*_\alpha) \colon ~\alpha \in u_t, ~\vp_\alpha \in r_t, ~ t < t_*  \}  \]
is a consistent partial type which does not fork over $M_*$. 

\br

\br

\noindent 
\emph{Case 2: $P(x) \in p$.}  Since $p$ is a type, there will already be $\eta_* \in \lim (\mct)$ such that $p \vdash P_{\eta \rstr k}(x)$ for all $k<\omega$. 
This information will be part of $p \rstr M_*$.  So only a single color is needed. 

That is, suppose we are given a sequence $\overline{\xr}$ of elements of $\mcr$, on which $G$ is constant and equal to $\beta$, and 
$\overline{b}^* = \langle b^*_\alpha : \alpha \in \bigcup_{t<t_*} w_t \rangle$ which is a good instantiation of $\overline{\xr}$. 
As we have ruled out trivial inconsistency, by our observation, 
\[ \{ \vp_\alpha(x, b^*_\alpha) \colon ~\alpha \in u_t, ~\vp_\alpha \in r_t, ~ t < t_*) \} \cup \{ P_{\eta_* \rstr k}(x) : k < \omega \}   \]
is a consistent partial type which does not fork over $M_*$, and this suffices. 
\end{proof}

\begin{disc}
It is still interesting to ask what would happen if we had available only finitely much information from $p \rstr M_*$. Would some coloring work, which does not rely on having already 
determined the predicates $Q_S$ or $P_\nu$?  In the remainder of this section we consider this, which will give the key idea for dealing with ultrapowers in the next section.  We handle just the case of $Q(x) \in p$ as an illustration, since both cases are worked out in detail in the next section. 
\end{disc}

For each $a \in P^{\mathfrak{C}_{\tsf}}$, let $\leaf(a)$ denote the ``leaf of $a$,'' i.e. the unique $\eta \in \lim (\mct)$ such that 
$\models P_{\eta \rstr k}(a)$ for all $k<\omega$.  

\begin{defn}
Call $B \subsetneq \lim (\mct)$ a \emph{blocking set} when: for every $A \subseteq \mathfrak{C}_{\tsf}$, 
\[ \mbox{ if $\{ \leaf(a) : a \in A \} \cap B = \emptyset$ then $\{ R(x,a) : a \in A \}$ is a partial type. } \] 
\end{defn}
[We can also give a direct definition, using \ref{c2}: $B$ is a blocking set if $\emptyset \subsetneq B \subsetneq \lim (\mct)$, and for all $\eta \in B$ and $k<\omega$ and $\ell \leq f(k)$,  there is $\eta^\prime \in B$ such that 
$\eta \rstr k ~^\smallfrown \langle \ell \rangle \tlf \eta^\prime$.]  
The number of blocking sets is no more than $2^{2^{\aleph_0}}$.\footnote{But see Comment \ref{continuum}.} 
Let $\bar{B} = \langle B_\alpha : \alpha < 2^{2^{\aleph_0}} \rangle$ 
enumerate them, possibly with repetition. 
Now for each $\xr = (u,w,q,r) \in \mcr$, and each $\alpha \in w$, as $q$ is a type, each $x_\alpha$ is either determined to belong to $P$ or to $Q$. 
If $x_\alpha$ belongs to $P$, then (again since $q$ is a type) there is $\eta_{\alpha} \in \lim (\mct)$ such that $q \vdash P_{\eta_{\alpha} \rstr k}(x_\alpha)$ for 
$k<\omega$.  
Moreover, as $r$ is a type, there is at least one blocking set $B$ such that 
\[ \{ \leaf_\xr(x_\alpha) : \alpha \in w, R(x,x_\alpha) \in q \} \cap B = \emptyset.  \] 
Let $\beta_\xr$ be an index for this $B$ in the enumeration $\bar{B}$, say for definiteness a minimal index. 
Choose the coloring function $G$ so that $G(\xr) = \beta_\xr$ for all $\xr \in \mcr$. 

Let's verify that this works. Suppose we are given a sequence $\overline{\xr}$ of elements of $\mcr$, on which $G$ is constant and equal to $\beta$, and 
$\overline{b}^* = \langle b^*_\alpha : \alpha \in \bigcup_{t<t_*} w_t \rangle$ which is a good instantiation of $\overline{\xr}$. 
Since we have ruled out trivial inconsistency with \ref{d:good-inst}, 
inconsistency cannot come from direct disagreement in the sense that, say, $R(x,b^*_\alpha)$ appears in one instance and $\neg(R(x,b^*_\alpha))$ appears in another.  
It will suffice to show that the type restricted to positive instances of $R$ is consistent. 
By definition of $G$, $\{ \leaf(b^*_\alpha) : \alpha \in w_t, t < t_*, R(x,x_\alpha) \in q_t ~\} \cap B_{\beta} = \emptyset$, 
hence the set of formulas 
\[ \{ \vp_\alpha(x, b^*_\alpha) \colon ~\alpha \in u_t, ~\vp_\alpha \in r_t, ~ t < t_* , \vp_\alpha(x,y) = R(x,y) \}  \]
is a consistent partial type which does not fork over $M_*$, and this suffices.

\br
\begin{comm} \label{continuum}
In fact, as the theory has trivial forking, we may use $\sigma = \aleph_0$, $\theta = \aleph_1$ and various natural changes to the definition 
to accommodate this, such as having the closure of a set be itself; see Observation 3.5 of \cite{MiSh:1030} and the paragraph before it. 
With these modifications, we can use co-finite blocking sets only, hence we can replace $2^{2^{\aleph_0}}$ above by $2^{\aleph_0}$. 
We plan to address this in future work $($but see also the proof of Theorem \ref{theta}$)$. 
\end{comm}

\vspace{5mm}

\section{A saturation result for $\tsf$} \label{s:sat}
\setcounter{equation}{0}

\begin{obs}[see e.g. Jech Theorem I.5.20]  \label{o:values}
To satisfy Definition $\ref{d:suitable}$, we 
may take e.g. $\sigma = \theta = \aleph_1$, $\mu = 2^{{\aleph_0}}$ and $\lambda = \mu^{+\ell}$ for any finite $\ell >0$.  
\end{obs}

Perfect ultrafilters were defined and shown to exist in \cite{MiSh:1030} \S 9, for the case of suitable $(\lambda, \mu, \aleph_0, \aleph_0)$.  
[These were called $(\lambda, \mu)$-perfect, with $\theta$, $\sigma$ omitted when countable.] 
Here we make the essentially cosmetic changes to extend this definition to allow for possibly uncountable $\theta$, starting with the definition. 

\begin{defn}[Support of a sequence, \cite{MiSh:1030} Definition 5.6.1] \label{d:support}
Let $\overline{\mb} = \langle \mb_u : u \in [\lambda]^{<\aleph_0} \rangle$ be a sequence of elements of $\ba = \ba^1_{2^\lambda, \mu, \theta}$. 
We say $X$ is a support of $\overline{\mb}$ in $\ba$ when $X \subseteq \{ \mx_f : f \in \fin_{\mu, \theta}(\alpha) \}$ and 
for each $u \in [\lambda]^{<\aleph_0}$ there is a maximal antichain of $\ba$
consisting of elements of $X$ all of which are either $\leq \mb_u$ or $\leq 1 -\mb_u$. 
Though there is no canonical choice of support we will write $\supp(\bar{\mb})$ to mean \emph{some} support. 
\end{defn}

\noindent Defintion \ref{d:perfect-0} extends \cite{MiSh:1030}, Definition 9.1 to possibly uncountable $\theta$. 

\begin{defn}[Perfect ultrafilters, for suitable $(\lambda, \mu, \theta, \aleph_0)$] 
\label{d:perfect-0} 
Let $(\lambda, \mu, \theta, \aleph_0)$ be suitable. 
We say that an ultrafilter $\de_*$ on $\ba = \ba^1_{2^\lambda, \mu, \theta}$ 
is \emph{$(\lambda, \mu, \theta, \aleph_0)$-perfect} when $(A)$ implies $(B)$:
\begin{enumerate}
\item[(A)] $\langle \mb_u : u \in [\lambda]^{<\aleph_0} \rangle$ is a monotonic sequence of elements of $\de_*$ 
\\ and  
$\supp(\bar{\mb})$ is a support for $\bar{\mb}$ of cardinality $\leq \lambda$, see $\ref{d:support}$, such that \\ for every $\alpha < 2^\lambda$ with 
$\bigcup \{ \dom(f) : \mx_f \in \supp(\overline{\mb}) \} \subseteq \alpha$, 
\\ there exists a multiplicative sequence 
\[ \langle \mb^\prime_u : u \in [\lambda]^{<\aleph_0} \rangle \]
of elements of $\ba^+$ 
such that
\begin{itemize} 
\item[(a)] $\mb^\prime_{u} \leq \mb_{u}$ for all $u \in [\lambda]^{<\aleph_0}$,  
\item[(b)] for every $\mc \in \ba^+_{\alpha, \mu, \theta} \cap \de_*$, 
no intersection of finitely many members of  
$\{ \mb^\prime_{\{i\}} \cup (1-\mb_{\{i\}}) : i < \lambda \}$ %\]
is disjoint to $\mc$. 
\end{itemize}
\item[(B)] there is a multiplicative sequence $\bar{\mb}^\prime = \langle \mb^\prime_u : u \in [\lambda]^{<\aleph_0} \rangle$ 
of elements of $\de_*$ which refines $\bar{\mb}$. 
\end{enumerate}
\end{defn}

\noindent Definition \ref{d:perfected} extends \cite{MiSh:1050}, Definition 3.11 to possibly uncountable $\theta$. 

\begin{defn} \label{d:perfected} 
Suppose $(\lambda, \mu, \theta, \aleph_0)$ are suitable. 
If $\de$ is built from from $(\de_0, \ba, \de_*)$ where $\de_0$ is a regular filter on $I$, $|I| = \lambda$, 
$\ba = \ba^1_{2^\lambda, \mu, \theta}$ and $\de_*$ is $(\lambda, \mu, \theta, \aleph_0)$-perfect, say $\de$ 
is \emph{$(\lambda, \mu, \theta, \aleph_0)$-perfected.}
\end{defn}

\noindent In the Appendix below, we update \cite{MiSh:1030} Theorem 9.4 to allow for the possibility of 
uncountable $\theta$: 

\begin{thm-e}[Existence, Theorem \ref{t:perfect-exists} below] 
Let $(\lambda, \mu, \theta, \aleph_0)$ be suitable. 
Let $\ba = \ba^1_{2^\lambda, \mu, \theta}$. Then there exists a $(\lambda, \mu, \theta, \aleph_0)$-perfect ultrafilter on $\ba$.
\end{thm-e}

The main result of this section is that perfect ultrafilters are able to saturate $\tsf$ for an uncountable but constant value of $\mu$. 

\begin{theorem} \label{theta}
Let $(\lambda, \mu, \theta, \aleph_0)$ be suitable, and in addition suppose $\mu \geq 2^{\aleph_0}$ and $\theta \geq \aleph_1$ 
$($e.g. let $\theta = \aleph_1$, $\mu = 2^{\aleph_0}$, and $\lambda = \mu^{+n}$ for any finite $n$$)$. 
Let $\de$ be a $(\lambda, \mu, \theta, \aleph_0)$-perfected ultrafilter on $I$, $|I| = \lambda$. 
Then $\de$ is good for $\tsf$, i.e. for any $M \models \tsf$, the ultrapower $M^I/\de$ is $\lambda^+$-saturated. 
\end{theorem}

\begin{proof}
We begin with the usual setup.  We fix $\de_0$, $\ba = \ba^1_{2^\lambda, \mu, \theta}$, $\jj$ and a $(\lambda, \mu, \theta, \aleph_0)$-perfect ultrafilter $\de_*$ on $\ba$ 
such that $\de$ is built from $(\de_0, \ba, \de_*)$ via $\jj$.
We choose $M \models \tsf$ as the index model, without loss of generality $\lambda^+$-saturated (by regularity of $\de$ the choice of $M$ will not matter). 
We fix some lifting from $M^I/\de$ to $M^I$, so that for each $a \in M^I/\de$ and each index $t \in I$ the projection $a[t]$ is well defined.  
If $\bar{c}  = \langle c_i : i < m \rangle \in {^m(M^I/\de)}$ then we use $\bar{c}[t]$ to denote $\langle c_i[t] : i < m \rangle$. 

Following Claim \ref{up-types},  it suffices to consider partial types of the following form.  (Moreover, since $R$ is not symmetric, $Q(x)$ and $P(x)$ are 
implied by the rest of the partial type in each case, so we may omit them.)

\begin{enumerate}
\item[(1)] $\{Q(x) \} \cup  \{ R(x,a) : a \in A \}$ for $A \subseteq P^N$, $|A| \leq \lambda$. 

\item[(2)] $\{ P(x) \} \cup \{ R(b,x) : b \in B \}$ for $B \subseteq Q^N$, $|B| \leq \lambda$. 
\end{enumerate}
Fix a partial type $p = p(x)$ which is either of type (1) or type (2).  Depending on which there will 
be some minor choices to make in the proof below. 
Recall two useful facts from the proof of Claim \ref{c5} above: in models of $\tsf$, for each finite $n$, 
\begin{quotation}
\begin{itemize}
\item[(Fact A)] For elements $a_0, \dots, a_n \in N$, 
$N \models (\exists x)\bigwedge_{i\leq n}R(x,a_i)$ if and only if there exist $\eta_0, \dots, \eta_n \in \mct_{k_*}$ such that 
\\ $N \models (\exists x)\bigwedge_{i\leq n}(~R(x,a_i) \land P_{\eta_i}(a_i)~)$, \\ where $k_*$ is minimal such that $f(k_*) > n$. 
\item[(Fact B)] For elements $b_0, \dots, b_n \in N$, $N \models (\exists x)\bigwedge_{i\leq n}R(b_i, x)$ 
if and only if there exists $\eta \in \mct_{k_*}$ such that \\ $N \models (\exists x)\bigwedge_{i\leq n}(~R(b_i, x) \land P_{\eta}(x)~)$, 
\\ where $k_*$ is minimal such that $f(k_*) > n$. 
\end{itemize}
\end{quotation}

\br
\br
\noindent We'll follow the strategy of \cite{MiSh:1050}, Theorem 4.1.   

\br
\noindent \underline{We begin with the case where $p$ is of type (1).}  

\br
\noindent Without loss of generality (possibly $|A| < \lambda$,  but $||N|| \geq 2^\lambda$ so this is no problem), 
\begin{equation}
\mbox{ let $\langle a_i : i <\lambda \rangle$ list the elements of $A$ without repetition. }
\end{equation}
This induces an enumeration of $p$ as 
\begin{equation} \label{e1}
\langle R(x, a_i) : i < \lambda \rangle. 
\end{equation}
As usual, for each finite $u \subseteq \lambda$, let
\begin{equation} \label{e2}
B_u := \{ t \in I : M \models \exists x \bigwedge_{i \in u} R(x,a_i[t]) \} ~~\mbox{ and } \mb_u = \jj(B_u). 
\end{equation}
and let 
\begin{equation} 
\bar{\mb} = \langle \mb_u : u \in [\lambda]^{<\aleph_0} \rangle.
\end{equation}
First we build an appropriate support for $\bar{\mb}$.  This will require handling equality and leaves. For equality, for each $i,j \in \lambda$ let 
\begin{equation}
\label{c-equal}
A_{a_i=a_j} := \{ t \in I : a_i[t] = a_j[t] \} \mbox{ and let } \ma_{a_i = a_j} := \jj(A_{a_i=a_j}). 
\end{equation}
For leaves, for each $i \in \lambda$, and each $\eta \in \mct$, let 
\begin{equation} \label{e3}
\ma_{P_\eta(a_i)} = \jj(~ \{ t \in I :  M \models P_\eta(a_i[t]) \} ~).
\end{equation}
Remembering that $\theta > \aleph_0$, for each $\eta \in \lim (\mct)$, define 
\begin{equation} \label{a-leaf}
\ma_{\mbox{``}\leaf(a_i) = \eta\mbox{''}} := \bigcap_{k < \omega} \ma_{P_{\eta \rstr k}(a_i)}. 
\end{equation}
Then (\ref{a-leaf}) will be nonzero for some, but not necessarily all, $\eta$, however, for each $i$, 
\begin{equation} \label{e:ip}
\langle \ma_{\mbox{``}\leaf(a_i) = \eta\mbox{''}} : \eta \in \lim (\mct) \rangle
\end{equation}
is a maximal antichain of $\ba$. 

For each $i < \lambda$ let $\eff_{\{i\}}$ be the set of all $f \in \fin_{\mu, \theta}(2^\lambda)$ such that for some $j \leq i$, the three condtions (\ref{c:a1-1}), (\ref{c:a2-1}), (\ref{c:a3-1}) hold: 
\begin{equation} \label{c:a1-1} 
\mx_f \leq  \ma_{a_i=a_j}.
\end{equation} 
\begin{equation}
\label{c:a2-1} \mbox{for all $k<j$, } ~\mx_f \cap \ma_{a_i=a_k} ~ = 0.
\end{equation} 
\begin{equation}
\label{c:a3-1} \mbox{for some $\eta$, } ~\mx_f \leq \ma_{\mbox{``} \leaf(a_i) = \eta\mbox{''} }. 
\end{equation} 
For each finite $u \subseteq \lambda$, define $\eff_u$ to be $\bigcap \{ \eff_{\{i\}} : i \in u \}$. 
Each $\eff_u$ is upward closed, i.e. $f \in \eff_u$ and $g \in \fin_{\mu,\theta}(2^\lambda)$ and $g \supseteq f$ implies $g \in \eff_u$. 

For each $u \in [\lambda]^{<\aleph_0}$, by induction on $\zeta < \lambda$, choose a maximal antichain $\langle \mx_{f_\epsilon} : \epsilon < \epsilon_* \rangle$ of elements of $\ba$ such that 
(i) each $\mx_{f_\epsilon}$ is either $\leq \mb_u$ or $\leq 1 - \mb_u$, and (ii) each $f_\epsilon \in \eff_u$ and $0 \in \dom(f_\epsilon)$.  
Necessarily the construction will stop at an ordinal $< \mu^+$, but $\geq \mu$ because $0 \in \dom(f_\epsilon)$. Re-index this 
antichain as 
\begin{equation}
\langle \mx_{f_{u,\zeta}} : \zeta < \mu \rangle.
\end{equation}
Then 
\begin{equation} \label{the-support}
 \{ ~ \mx_{f_{u,\zeta}} : \zeta < \mu,  u \in [\lambda]^{<\aleph_0} ~\}
\end{equation}
is a support of $\bar{\mb}$ in the sense of Definition \ref{d:support}.  When $u = \{ i \}$, we will often write 
\[ f_{i, \zeta} \mbox{ to mean } f_{\{i\},\zeta}. \]

\br

Second, we build a multiplicative refinement for $\bar{\mb}$. 
\begin{equation}
\mbox{ Fix $\alpha < 2^\lambda$ such that $\bigcup \{ \dom(f_{u,\zeta}) : \zeta < \mu, u \in [\lambda]^{<\aleph_0} \} \subseteq \alpha$ }. 
\end{equation}
As before,  let $\leaf(a)$ denote the unique $\eta \in \lim (\mct)$ such that $\models P_{\eta \rstr k}(a)$ for all $k<\omega$, and 
call $X \subsetneq \lim (\mct)$ a \emph{blocking set} when: for every $A \subseteq P^{\mathfrak{C}_{\tsf}}$, 
\[ \mbox{ if $\{ \leaf(a) : a \in A \} \cap X = \emptyset$ then $\{ R(x,a) : a \in A \}$ is a partial type. } \] 
As $\mu \geq  2^{{\aleph_0}}$, let 
\begin{equation}
\langle X_\epsilon : \epsilon < \mu \rangle
\end{equation}
be an enumeration, possibly with repetitions, of all co-finite blocking sets.\footnote{A priori we could use all blocking sets and $\mu = 2^{2^{\aleph_0}}$, but the nice point is that in our present setup the co-finite blocking sets suffice. Note that here $p(x)$ being a set of \emph{positive} instances of $R(x,y)$ helps. With negation, we'd 
need $f_{u,\zeta}$'s deciding all cases of $R(x,a_i)$, $\neg R(x,a_i)$, $a_i =a_j$.} 
Let $H$ be the function from $\{ f_{i,\zeta} : i < \lambda, \zeta < \mu \} \times \mu$ to $ \{ 0, 1 \}$ given by 
\begin{equation}
 H (f_{i, \zeta}, \epsilon) = 1 \mbox{ iff } \eta \notin X_\epsilon   
\end{equation}
where $\eta$ is the unique element of $\lim (\mct)$ such that $\mx_{f_{i,\zeta}} \leq \ma_{\mbox{``}\leaf(a_i) = \eta\mbox{''}}$. 
(Very informally, $H$ returns 1 if a type avoiding the blocking set $B_\epsilon$ may contain $a_i$ as it appears on $\mx_{f_{i,\zeta}}$.) 

We'll need a new antichain to help us divide the work: 
\begin{equation}
\mbox{let $\bar{\mc} = \langle \mc_\epsilon : \epsilon < \mu \rangle$ be given by $\mc_{\epsilon} = \mx_{\{ (\alpha, \epsilon) \}} \}$.}
\end{equation}
Any element of this antichain will have nonzero intersection with any of the elements from $\ba^+_{\alpha, \mu, \theta}$, our protagonists so far.  

Finally, for each $i < \lambda$, define
\begin{equation} \label{c:one-1} 
\mb^\prime_{\{i \}} = \left ( \bigcup \{   \mx_{f_{i,\zeta}} \cap \mc_{ \epsilon   }           : \zeta < \mu  \mbox { and } H(f_{i,\zeta}, \epsilon) = 1   \}   \right) \cap \mb_{\{i\}}. 
\end{equation}
Define 
\begin{equation}
\mb^\prime_u = \bigcap_{i \in u} \mb^\prime_{\{i\}}, \mbox{ and let } \bar{\mb}^\prime = \langle \mb^\prime_u : u \in [\lambda]^{<\aleph_0} \rangle. 
\end{equation}
By definition, $\bar{\mb}^\prime$ is multiplicative. 
Our final task is to make sure the hypotheses of Definition \ref{d:perfect-0} are satisfied, i.e. that for our multiplicative sequence 
$\bar{\mb}^\prime$, 
\begin{itemize} 
\item[(a)] $\mb^\prime_{u} \leq \mb_{u}$ for all $u \in [\lambda]^{<\aleph_0}$,  
\item[(b)] for every $\mc \in \ba^+_{\alpha, \mu, \theta} \cap \de_*$, 
no intersection of finitely many members of  
\\ $\{ \mb^\prime_{\{i\}} \cup (1-\mb_{\{i\}}) : i < \lambda \} $
is disjoint to $\mc$. 
\end{itemize}
For (a), suppose for a contradiction that for some $u \in [\lambda]^{<\aleph_0}$ there were a nonzero 
\[ \mc \leq  \mb^\prime_u \setminus \mb_u . \]
Without loss of generality,  
\begin{equation}
\label{ce-1}
\mc \leq \mc_\epsilon \mbox{ for some } \epsilon < \mu,  
\end{equation}
and also, since $u$ is finite, 
\begin{equation}
\label{wlog}
\mbox{ $\mc$ is either below or disjoint to all elements in 
$\{ \mx_{f_{i,\zeta}} : i \in u, \zeta < \mu \}$}.
\end{equation}  
So for each $i\in u$ there is $\zeta_i < \mu$  with $\mc \leq \mx_{f_{i,\zeta_i}}$. By (\ref{ce-1}) and the definition (\ref{c:one-1}), 
\begin{equation}
\mc \leq \bigcap_{i \in u} \mx_{f_{i,\zeta_i}} \mbox{ and }  \bigwedge_{i \in u}  H(f_{i,\zeta_i}, \epsilon) = 1. 
\end{equation}
Now by our hypothesis, $\mc \leq \mb_{\{i\}}$, meaning
\begin{equation}
\mc \leq \ma[\exists x R(x,a_i)] ~~ \mbox{ for $i \in u$ }. 
\end{equation}
Let $k_*$ be minimal so that $f(k_*) > |u|$. Then as $H(f_{i,\zeta_i}, \epsilon) = 1$ for $i \in u$, 
\begin{equation}
\mc \leq \bigcap \{ \ma[\neg P_{\rho \rstr k}(a_i)]  : i \in u, \rho \in X_\epsilon, k \leq k_* \}. 
\end{equation}
Informally, on $\mc$ none of the parameters $a_i$ fall into the predicates forbidden by the blocking set, at least up to level $k_*$ (this suffices for our contradiction, recalling Fact A from the beginning of the proof). 
And $\mc \cap \mb_u = \emptyset$ means 
\begin{equation}
\mc \leq \ma[\neg \exists x \bigwedge_{i \in u} R(x,a_i)]. 
\end{equation}
Since $\bar{\mb}$ is a possibility pattern and $\mc > 0$, this means we should be able to find values for $a_i$ in $\mathfrak{C}_{\tsf}$ which would make this combination of 
formulas true, but this is impossible because $X_\epsilon$ is a blocking set (so avoiding it gives a type).  This contradiction shows that (a) holds, so $\bar{\mb}^\prime$ is a multiplicative 
refinement of $\bar{\mb}$. 

\br
For (b), it will suffice to show that for any $\ma \in \de_*$ such that $\supp(\ma) \subseteq \alpha$, and any finite $u \subseteq \lambda$,
\[ \ma \cap \bigcap \{ \mb^\prime_{\{i\}} \cup (1 - \mb_{\{i\}}) : i \in u \} > 0. \]
Without loss of generality, we can write $u = v \cup w$ where for each $i \in v$, $\ma \leq 1 - \mb_{\{i\}}$ and for each 
$i \in w$, $\ma \leq \mb_{\{i\}}$. If $w$ is empty we are done, so suppose not, and it will suffice to show that 
\[ \ma \cap \bigcap \{ \mb^\prime_{\{i\}} : i \in w \} > 0. \]
As $\mb_w \in \de_*$, without loss of generality $\ma \leq \mb_w$, and we may choose 
$g \in \fin_{\mu,\theta}(\lambda)$ such that $\mx_f \leq \ma$.  Moreover, for each $i \in w$, we may assume that there is some $\zeta_i < \mu$ such that 
$\mx_g \leq \mx_{f_{i,\zeta_i}}$. So we have that
\[ 0 < \mx_g \leq  \bigcap_{i \in w} \mx_{f_{i,\zeta_i}}  \leq \mb_w.  \]
Recall that by our choice of partitions, for each $f_{i,\zeta_i}$, there is a unique $\eta = \eta_i \in \lim (\mct)$ such that $\mx_{f_{i,\zeta}} \leq \ma_{\mbox{``}\leaf(a_i) = \eta\mbox{''}}$. 
Because this intersection is $\leq \mb_w$, there is some blocking set $X_\epsilon$ such that $X_\epsilon \cap \{ \eta_i : i \in w \} =\emptyset$. Then 
\[ 0 < \mc_\epsilon \cap \bigcap_{i \in w} \mx_{f_{i,\zeta_i}} \cap \mx_g \leq \bigcap \{ \mb^\prime_{\{i\}} : i \in w \} \]
which completes the proof of (b). 
This completes the proof of Case 1. 

\br

\noindent \underline{For case (2)}, the strategy is similar, with a few changes to reflect the dual type. For clarity we give the entire argument, renaming the parameter set as $B$.  

\setcounter{section}{6}
\setcounter{equation}{0}

\br \noindent Without loss of generality, 
\begin{equation}
\mbox{ let $\langle b_i : i <\lambda \rangle$ list the elements of $B$ without repetition. }
\end{equation}
This induces an enumeration of $p$ as 
\begin{equation} \label{e1}
\langle R(b_i, x) : i < \lambda \rangle. 
\end{equation}
For each finite $u \subseteq \lambda$, let
\begin{equation} \label{e2}
B_u := \{ t \in I : M \models \exists x \bigwedge_{i \in u} R(b_i[t], x) \} ~~ \mbox{ and } \mb_u = \jj(B_u). 
\end{equation}
and let 
\begin{equation} 
\bar{\mb} = \langle \mb_u : u \in [\lambda]^{<\aleph_0} \rangle.
\end{equation}
First we build an appropriate support for $\bar{\mb}$.  As before, for each $i,j \in \lambda$ let 
\begin{equation}
\label{c-equal}
A_{b_i=b_j} := \{ t \in I : b_i[t] = b_j[t] \} \mbox{ and let } \ma_{b_i = b_j} := \jj(A_{b_i=b_j}). 
\end{equation}
Now for each $\eta \in \mct$, 
\begin{equation} \label{e3} 
\ma_{(\exists x)(R(b_i,x) \land P_\eta(x))} = \jj(~ \{ t \in I : M \models \exists x (R(b_i[t], x) \land P_\eta(x)) \} ~).  
\end{equation}
As $\theta > \aleph_0$, for each $i < \lambda$ and $\eta \in \lim (\mct)$, define 
\begin{equation} \label{a-set}
\ma_{\mbox{``}(\exists x)(R(b_i,x) \land \leaf(x) = \eta)\mbox{''}} := \bigcap_{k < \omega} \ma_{(\exists x)(R(b_i,x) \land P_{\eta\rstr k}(x))}. 
\end{equation}
For each $s \in \mcs$, 
\[ \ma_{Q_s(b_i)} = \jj ( ~ \{ t \in I : M \models Q_s(b_i[t]) \} ~). \]
For each $\nu \in \lim(\mcs)$, letting $\nu = \langle s_k : k < \omega \rangle$ (recalling \ref{n:ds}) so $\nu(k)$ denotes $s_k$, define 
\[ \ma_{\mbox{``}Q_{\nu}(b_i)\mbox{''}} := \bigcap_{k<\omega} \ma_{Q_{\nu(k)}(b_i)} \]
Then for each $i < \lambda$, 
\begin{equation}
\label{e:np}
\langle \ma_{\mbox{``}Q_{\nu}(b_i)\mbox{''}} : \nu \in \lim (\mcs) \rangle 
\end{equation} 
is a partition of $\mb_{\{i \}}$. 
For each finite $u \subseteq \lambda$ let $\eff_{u}$ be the set of all $f \in \fin_{\mu, \theta}(2^\lambda)$ such that the conditions (\ref{c:a1}), (\ref{c:a3}) hold: 
\begin{equation} \label{c:a1} 
\mbox{ for $i \in u$, for some $j \leq i$, }
\mx_f \leq  \ma_{b_i=b_j} \mbox{ and } \mbox{for all $k<j$, } ~\mx_f \cap \ma_{b_i=b_k} ~ = 0.
\end{equation} 
\begin{equation}
\label{c:a3} \mbox{ for $i \in u$, for some $\nu \in \lim (\mcs)$, }
\mx_f  \leq \ma_{  \mbox{``}Q_{\nu}( b_i )   \mbox{''}    }. 
\end{equation}
It follows that for any $f \in \eff_{\{i\}}$, if $\nu \in \lim (\mcs)$ is such that $\mx_f  \leq \ma_{\mbox{``}Q_{\nu}(b_{\{i\}})\mbox{''}}$, then 
for some $\eta \in \lim (\mct)$, indeed for any $\eta$ such that $\eta \rstr k \in \nu(k)$ for all $k<\omega$, 
\begin{equation} \label{e:leaf} ~\mx_f \leq %\bigcap_{i \in u} 
\ma_{\mbox{``}(\exists x)(R(b_i,x) \land \leaf(x) = \eta)\mbox{''}}. 
\end{equation}
Each $\eff_u$ is upward closed, i.e. $f \in \eff_u$ and $g \supseteq f$ implies $g \in \eff_u$. 

For each $u \in [\lambda]^{<\aleph_0}$, by induction on $\zeta < \lambda$, choose a maximal antichain $\langle \mx_{f_\epsilon} : \epsilon < \epsilon_* \rangle$ of elements of $\ba$ such that 
(i) each $\mx_{f_\epsilon}$ is either $\leq \mb_u$ or $\leq 1 - \mb_u$, and (ii) each $f_\epsilon \in \eff_u$.  Necessarily the construction will stop at an ordinal $< \mu^+$. 
Renumber this 
antichain as 
\[ \langle \mx_{f_{u,\zeta}} : \zeta < \mu \rangle. \]
Then 
\begin{equation} \label{the-support}
 \{ ~ \mx_{f_{u,\zeta}} : \zeta < \mu,  u \in [\lambda]^{<\aleph_0} ~\}
\end{equation}
is a support of $\bar{\mb}$ in the sense of Definition \ref{d:support}. %, and also satisfies (\ref{e:coh}). 

When $u = \{ i \}$, we will again write 
\[ f_{i, \zeta} \mbox{ to mean } f_{\{i\},\zeta}. \]

\br

Second, we build a multiplicative refinement for $\bar{\mb}$. 
\[ \mbox{ Fix $\alpha < 2^\lambda$ such that $\bigcup \{ \dom(f_{u,\zeta}) : \zeta < \mu, u \in [\lambda]^{<\aleph_0} \} \subseteq \alpha$ }. \] 
As $\mu \geq  2^{{\aleph_0}}$, let 
\begin{equation} \label{enum-eta}
\langle \eta_\epsilon : \epsilon < \mu \rangle
\end{equation}
be an enumeration, possibly with repetitions, of all leaves $\eta \in \lim (\mct)$. 
Let $G$ be the function from $\{ f_{i,\zeta} : i < \lambda, \zeta < \mu \} \times \mu$ to $ \{ 0, 1 \}$ given by 
\begin{equation}
 G (f_{i, \zeta}, \epsilon) = 1 \mbox{ iff \hspace{5mm}} \mx_{f_{i,\zeta}} \leq \ma_{\mbox{``}Q_{\nu}(b_i)\mbox{''}} \mbox{ and $\eta_\epsilon \rstr k \in \nu(k)$ for all $k<\omega$}.  
\end{equation}
We'll need a new antichain to help us divide the work: 
\begin{equation}
\mbox{let $\bar{\mc} = \langle \mc_\epsilon : \epsilon < \mu \rangle$ be given by $\mc_{\epsilon} = \mx_{\{ (\alpha + 1, \epsilon) \}} \}$.}
\end{equation}
Any element of this antichain will have nonzero intersection with any of the elements from $\ba^+_{\alpha, \theta}$. 

For each $i < \lambda$, define
\begin{equation} \label{c:one} 
\mb^\prime_{\{i \}} = \left ( \bigcup \{   \mx_{f_{i,\zeta}} \cap \mc_{ \epsilon   }           : \zeta < \mu  \mbox { and } G(f_{i,\zeta}, \epsilon) = 1   \}   \right) \cap \mb_{\{i\}}. 
\end{equation}
Define 
\[ \mb^\prime_u = \bigcap_{i \in u} \mb^\prime_{\{i\}}, \mbox{ and let } \bar{\mb}^\prime = \langle \mb^\prime_u : u \in [\lambda]^{<\aleph_0} \rangle. \]
By definition, $\bar{\mb}^\prime$ is multiplicative.  Again we address the hypotheses of Definition \ref{d:perfect-0}.
For (a), suppose for a contradiction that for some $u \in [\lambda]^{<\aleph_0}$ there were a nonzero 
\[ \mc \leq  \mb^\prime_u \setminus \mb_u . \]
Without loss of generality,  
\begin{equation}
\label{ce}
\mc \leq \mc_\epsilon \mbox{ for some } \epsilon < \mu,  
\end{equation}
and also 
\begin{equation}
\label{wlog}
\mbox{ $\mc$ is either below or disjoint to all elements in 
$\{ \mx_{f_{i,\zeta}} : i \in u, \zeta < \mu \}$}.
\end{equation}  So for each $i\in u$ there is some $\zeta_i < \mu$  with $\mc \leq \mx_{f_{i,\zeta}}$, %[and by (\ref{wlog}), $\zeta_i$ is unique] 
and then by (\ref{ce}) and the definition (\ref{c:one}), 
\begin{equation}
\mc \leq \bigcap_{i \in u} \mx_{f_{i,\zeta_i}} \mbox{ and }  \bigwedge_{i \in u}  G(f_{i,\zeta_i}, \epsilon) = 1. 
\end{equation}
Now by our hypothesis, $\mc \leq \mb_{\{i\}}$, meaning
\begin{equation} \label{c-1}
\mc \leq \ma[\exists x R(b_i, x)] ~~ \mbox{ for $i \in u$ }. 
\end{equation}
And $\mc \cap \mb_u = \emptyset$ means 
\begin{equation} \label{c-2}
\mc \leq \ma[\neg \exists x \bigwedge_{i \in u} R(b_i,x)]. 
\end{equation}
Recalling the definition of $G$ and (\ref{ce}), for every $i \in u$, 
\begin{equation} \label{c-3}
\mc \leq \mx_{f_{i,\zeta}} \leq \ma_{\mbox{``}(\exists x)(R(b_i,x) \land \leaf(x) = \eta_\epsilon)\mbox{''}}. 
\end{equation}
But (\ref{c-1}), (\ref{c-2}), and (\ref{c-3}) together are a contradiction.  [Why? (\ref{c-2}) must be witnessed by full splitting at some finite stage, 
but (\ref{c-3}) guarantees that at every finite stage there is a specific piece of the successor partition which is reserved for $x$.] 
This contradiction shows that condition (a) of the definition of perfect holds, so $\bar{\mb}^\prime$ is indeed a multiplicative 
refinement of $\bar{\mb}$. 

\br
For (b), it will suffice to show that for any $\mc \in \de_*$ such that $\supp(\ma) \subseteq \alpha$, and any finite $u \subseteq \lambda$,
\[ \mc \cap \bigcap \{ \mb^\prime_{\{i\}} \cup (1 - \mb_{\{i\}}) : i \in u \} > 0. \]
Without loss of generality, we can write $u = v \cup w$ where for each $i \in v$, $\mc \leq 1 - \mb_{\{i\}}$ and for each 
$i \in w$, $\mc \leq \mb_{\{i\}}$. If $w$ is empty we are done, so suppose not, and it will suffice to show that 
\[ \mc \cap \bigcap \{ \mb^\prime_{\{i\}} : i \in w \} > 0. \]
As $\mb_w \in \de_*$, without loss of generality 
\[ \mc \leq \mb_w. \]
Recalling Fact B from the beginning of the proof, let $k_*$ be minimal so that $f(k_*) > |w|$, and then 
\[ \bigcup \{ \ma[ ~(\exists x)\bigwedge_{i \in w} R(b_i,x) \land P_\eta(x)~] ~:~ \eta \in \mct_{k_*} \} = \mb_w \]
so, after possibly shrinking $\mc$ by taking intersections, we may assume there is $\eta_* \in \mct_{k_*}$ such that 
\[ \mc \leq  \ma[(\exists x)\bigwedge_{i \in w} R(b_i,x) \land P_{\eta_*}(x)]. \]
Again by Fact B, this implies there is some $\eta \in \lim (\mct)$ such that $\eta_* \tlf \eta$ and for all finite $k$, 
\[ \mc \cap  \ma[(\exists x)\bigwedge_{i \in w} R(b_i,x) \land P_{\eta \rstr k}(x)] > 0\]
so as $\theta > 0$, without loss of generality 
\[ \mc \leq \ma_{  \mbox{``}    (\exists x)(R(b_i,x) \land \leaf(x) = \eta)   \mbox{''}    }. \]
Recalling the support (\ref{the-support}), after possibly shrinking $\mc$ by taking intersections, there are $\zeta_i$ for each $i \in w$ such that 
\[ \mc \leq \mx_{f_{i,\zeta}} \]
i.e. for each $i \in w$, 
\[ 0 <  \mx_{f_{i,\zeta_i}} \cap \mc \leq \ma_{  \mbox{``}    (\exists x)(R(b_i,x) \land \leaf(x) = \eta)   \mbox{''}    }. \]
Our choice of partitions in (\ref{c:a3}) and (\ref{the-support}) means that for each $i \in w$ there is $\nu_i \in \lim (\mcs)$ such that 
\[  \mx_{f_{i,\zeta_i}} \leq \ma_{``Q_{\nu_i}(b_i)''}. \]
Recalling (\ref{e:leaf}), we conclude from these two equations that $\eta \rstr k \in \nu_i(k)$ for all $k<\omega$, for each $i \in w$. 
So letting $\epsilon$ be such that $\eta = \eta_\epsilon$ in the enumeration (\ref{enum-eta}), $G(f_{i,\zeta_i}, \epsilon) = 1$ for all $i \in w$. 
We have shown that 
\[ \mc \cap \bigcap_{i \in w} \mx_{f_{i,\zeta_i}}  \cap \mc_\epsilon > 0 \]
and this suffices. 

\br
\noindent This completes Case 2, and so completes the proof of the Theorem. 
\end{proof}

\vspace{5mm}

\section{Consequences for Keisler's order}

\begin{theorem} Let $f \in \mcf$. 
In Keisler's order, $\tsf$ is strictly above the theory of the random graph.
\end{theorem}

\begin{proof}
Recall that $\trg$ is minimum in Keisler's order among the unstable theories (\cite{mm4}, Conclusion 5.3). 
So as $\tsf$ is unstable, $\trg \trianglelefteq \tsf$. 
By Lemma \ref{c8}, if $\de$ is a regular ultrafilter on $\lambda$ built from $(\de_0, \ba = \ba^1_{2^\lambda, \aleph_0, \aleph_0}, \de_*)$ where 
$\de_0$ is any regular good [or so-called excellent] filter on $\lambda$ and $\de_*$ is any ultrafilter on $\ba$, then $\de$ is not good for $\tsf$. On the other hand, by 
\cite{MiSh:1009} Theorem 3.2 in the case $\mu = \aleph_0$, there is such an ultrafilter which is good for $\trg$. 
This shows that $\trg \triangleleft \tsf$. 
\end{proof}

We recall the higher analogues of the triangle-free random graph, studied by 
Hrushovski \cite{h:letter}. In particular, he proved that each $T_{n,k}$ is simple unstable with trivial forking for $n>k\geq 2$.

\begin{defn}  Recall that 
$T_{n,k}$ denotes the generic $(n+1)$-free $(k+1)$-hypergraph, i.e. the model completion of the theory $($in a language with a single 
$(k+1)$-place relation, interpreted as a hyperedge, so symmetric and irreflexive$)$ stating that there do not exist $(n+1)$ distinct elements 
of which every $(k+1)$ are an edge. 
\end{defn}

The infinite descending chain in Keisler's order obtained in \cite{MiSh:1050} was given by $\cdots T_m \tlfn T_{n-1} \tlfn \cdots T_1 \tlfn  T_0$ where 
$T_n$ is the disjoint union of the theories $T_{k+1,k}$ for $k > 2n+2$.

In the Appendix, Theorem \ref{theta-2} below we update \cite{MiSh:1050} Claim 5.1 to allow for the possibility of uncountable $\theta$. 

\begin{thm-e}[\cite{MiSh:1050}, Claim 5.1 for possibly uncountable $\theta$, Theorem \ref{theta-2} below] \label{theta-2}
Suppose that:  
\begin{enumerate}
\item for 
 integers $2 \leq k < \ell$, and e.g. $\theta = \aleph_1$, $\mu = 2^{\aleph_0}$, $\lambda = \mu^{+\ell}$, 
\\ or just: $(\lambda, k, \mu^+) \rightarrow k+1$ in the sense of \cite{MiSh:1050} Notation 1.2 
\item $\ba = \ba^1_{2^\lambda, \mu, \theta}$ 
\item $\de_*$ is an ultrafilter on $\ba$
\item $T = T_{k+1,k}$ 
\end{enumerate}
Then $\de_*$ is not $(\lambda, T)$-moral. 
\end{thm-e}

\begin{theorem} \label{t:izfc} For any finite $k \geq 2$, 
$\tsf$ and $T_{k+1,k}$ are incomparable in Keisler's order.  More precisely:

\begin{enumerate}
\item  let $\de$ be a $(\lambda, \mu, \aleph_0, \aleph_0)$-perfected ultrafilter on $\lambda$ where %$\ba_{2^\lambda, \mu, \theta}$ where 
\[ \lambda = \aleph_{k-1}, ~\mu = \aleph_0, ~\theta =\aleph_0. \]  
Then $\de$ is not good for $\tsf$, but 
it is good for $T_{k+1,k}$. 

\item let $\de_*$ be a $(\lambda, \mu, \aleph_1, \aleph_0)$-perfected ultrafilter on $\lambda$ where 
\[ \mu \geq 2^{{\aleph_0}} \mbox{ and } \lambda = \mu^{+n} \mbox{ for } n > k+1. \] 
Then $\de$ is good for $\tsf$, but it is not good for $T_{k+1,k}$. 
\end{enumerate}
\end{theorem}

\begin{proof}
(1) The non-saturation is Lemma \ref{c8} via separation of variables, and the saturation is \cite{MiSh:1050} Theorem 4.1. 

(2) The saturation is Theorem \ref{theta}, and the non-saturation is Theorem \ref{theta-2} via separation of variables. 
\end{proof}

\vspace{5mm}

\section*{Appendix: perfect ultrafilters for uncountable $\theta$}

\setcounter{section}{9}

In \cite{MiSh:1030},   we considered \emph{suitable} tuples of cardinals $(\lambda, \mu, \theta, \sigma)$, see Definition \ref{d:suitable}. 
We defined 
\[ \mbox{ ``$\de$ is a $(\lambda, \mu, \theta, \sigma)$-ultrafilter on the Boolean algebra $\ba^1_{2^\lambda, \mu, \theta}$'' } \]
in the case where $\theta = \sigma = \aleph_0$, and we proved that such ultrafilters did indeed exist. 

In this Appendix, we upgrade that definition and existence proof to include the case of uncountable $\theta$. The proof is almost word-for-word the same as that of \cite{MiSh:1030} \S 9, 
but to eliminate doubt, we have reproduced that proof here with the minor changes. 
We defined ``support'' in \ref{d:support} above and ``perfect'' in \ref{d:perfect-0} above. 

\begin{conv}
Throughout this section we assume:
\[ \lambda \geq \mu^{<\theta} \geq  \theta = \cf(\theta) \geq \sigma = \aleph_0. \]
Without loss of generality we may assume $\theta > \sigma$, as the case $\theta = \sigma = \aleph_0$ was the case of \cite{MiSh:1030} $\S 9$. 
\end{conv}

\begin{obs} \label{o:support}
Suppose $\alpha < 2^\lambda$ is fixed, $D_\alpha$ is an ultrafilter on $\ba^1_{\alpha, \mu, \theta} \subseteq \ba = \ba^1_{2^\lambda, \mu, \theta}$, and 
$\langle \mb_u : u \in [\lambda]^{<\aleph_0} \rangle$ is a sequence of elements of $D_\alpha$. Suppose 
that there exists a multiplicative sequence 
$\langle \mb^\prime_u : u \in [\lambda]^{<\aleph_0} \rangle$
of elements of $\ba^+$ 
such that
\begin{itemize} 
\item[(a)] $\mb^\prime_{u} \leq \mb_{u}$ for all $u \in [\lambda]^{<\aleph_0}$,  
\item[(b)] for every $\mc \in \ba^+_{\alpha, \mu, \theta} \cap \de_\alpha$, 
no intersection of finitely many members of  
$\{ \mb^\prime_{\{i\}} \cup (1-\mb_{\{i\}}) : i < \lambda \}$ %\]
is disjoint to $\mc$. 
\end{itemize}
Then there is a multiplicative sequence $\langle \mb^{\prime\prime}_u : u \in [\lambda]^{<\aleph_0} \rangle$ 
such that (a), (b) hold with $\mb^\prime_u$, $\mb^\prime_{\{i\}}$ 
replaced by $\mb^{\prime\prime}_u$, $\mb^{\prime\prime}_{\{i\}}$ respectively, and such that 
some support of $\bar{\mb^{\prime\prime}}$ is contained in $\ba_{\alpha + \lambda, \mu, \theta}$.
\end{obs}

\begin{proof} 
Without loss of generality there is $\mcv$ of cardinality $\lambda$ such that some support of $\bar{\mb}^\prime$ is contained in 
$\{ \mx_f : f \in \fin_{\mu, \theta}(\mcv) \}$. Let $\pi$ be a permutation of $2^\lambda$ which is the identity on $\alpha$ and takes $\uu$ into $\alpha + \lambda$.  
This induces an automorphism $\rho$ of $\ba$ which is the identity on $\ba_{\alpha, \mu, \theta}$, so in particular is the identity on 
$\de_\alpha$ and thus on $\bar{\mb}$. For each $u \in [\lambda]^{<\aleph_0}$, let $\mb^{\prime\prime}_u = \rho(\mb^\prime_u)$. 
Then clearly $\bar{\mb^{\prime\prime}}$ fits the bill. 
\end{proof}

\begin{theorem}[Existence] \label{t:perfect-exists}
Let $(\lambda, \mu, \theta, \aleph_0)$ be suitable.  
Let $\ba = \ba^1_{2^\lambda, \mu, \theta}$. Then there exists a $(\lambda, \mu, \theta, \aleph_0)$-perfect ultrafilter on $\ba$.
\end{theorem}

\begin{proof} 
Begin by letting $\langle \bar{\mb}_\delta = \langle \mb_{\delta, u} : u \in [\lambda]^{<\aleph_0} \rangle : \delta < 2^\lambda \rangle$ be an enumeration of the monotonic sequences of elements of $\ba^+$, each occurring cofinally often.  
Let $z: 2^\lambda \rightarrow 2^\lambda$ be an increasing continuous function which satisfies: $z(0) \geq 0$ and for all 
$\beta < 2^\lambda$,  $z(\beta) + \lambda = z(\beta + 1)$.
By induction on $\delta < 2^\lambda$ we will construct 
$\langle D_\delta: \delta < 2^\lambda \rangle$, an increasing continuous sequence of filters with 
each $D_\delta$ an ultrafilter on $\ba_{z(\delta), \mu, \theta}$, to satisfy:   

\begin{enumerate}
\item[(*)] \emph{if} $\delta = \beta + 1$, if it is the case that 

\begin{quotation}
\noindent $\langle \mb_{\beta, u} : u \in [\lambda]^{<\aleph_0} \rangle$ is a monotonic sequence of elements of $\de_\beta$ and   
there exists a choice of $\supp(\bar{b})$ with 
$\bigcup \{ \dom(f) : \mx_f \in \supp(\overline{\mb}) \} \subseteq \beta$
and there exists a multiplicative sequence 
\[ \langle \mb^\prime_u : u \in [\lambda]^{<\aleph_0} \rangle \]
of elements of $\ba^+$ 
such that
\begin{itemize} 
\item[(a)] $\mb^\prime_{u} \leq \mb_{\beta, u}$ for all $u \in [\lambda]^{<\aleph_0}$,  
\item[(b)] for every $\mc \in \ba^+_{z(\beta), \mu, \theta} \cap \de_\beta$, 
no intersection of finitely many members of 
$\{ \mb^\prime_{\{i\}} \cup (1-\mb_{\beta, \{i\}}) : i < \lambda \}$ %\]
is disjoint to $\mc$. 
\end{itemize}
\end{quotation}

\emph{then} there is a sequence $\bar{\mb}^{\prime\prime} = \langle {\mb}^{\prime\prime}_u : u \in [\lambda]^{<\aleph_0} \rangle$ 
of elements of $\ba^+$ such that:
\begin{enumerate}
\item[(i)] $\mb^{\prime\prime}_{u} \leq \mb_{\beta, u}$ for all $u \in [\lambda]^{<\aleph_0}$,  
\item[(ii)] for every $\mc \in \ba^+_{z(\beta), \mu, \theta} \cap \de_\beta$, 
no intersection of finitely many members of 
$\{ \mb^{\prime\prime}_{\{i\}} \cup (1-\mb_{\beta, \{i\}}) : i < \lambda \}$ %\]
is disjoint to $\mc$. 
\item[(iii)] some support of $\bar{\mb}^{\prime\prime}$ is contained in $\ba_{z(\delta), \mu}$, and 
\item[(iiv)] $D_\delta$ is an ultrafilter on $\ba_{z(\delta), \mu, \theta}$ which extends $D_\beta \cup \{ \mb^\prime_{u} : u \in [\lambda]^{<\aleph_0} \}$. 
\end{enumerate} 
\end{enumerate}
The induction may be carried out at limit stages because all of the $D_\delta$ are ultrafilters. 
Suppose $\delta = \beta + 1$. If $\bar{\mb}$ satisfies the quoted condition, then let 
$\bar{\mb}^{\prime\prime}$ be given by Observation \ref{o:support}, using $z(\beta)$ here for $\alpha$ there. 
Then (i), (ii), (iii) are satisfied, so we need to prove that 
\[ D_\beta \cup \{ \mb^{\prime\prime}_{u} : u \in [\lambda]^{<\aleph_0} \} \] 
has the finite intersection property. As $D_\beta$ is an ultrafilter on $\ba_{z(\beta), \mu, \theta}$, 
and $\bar{\mb}^\prime$ is a multiplicative sequence, it suffices to prove that for any 
$\mc \in \de_\beta$ and any finite $u \subseteq \lambda$, 
\[ \mc \cap \bigcap \{ \mb^{\prime\prime}_{\{i\}} : i \in u \} > 0. \]
As $\mb_{\{i\}} \in \de_\beta$ for each $i \in u$, we may assume that $\mc \cap (1-\mb_{\{i\}}) = 0$ for each $i \in u$. 
Then we are finished by assumption (ii). 
This completes the induction. Let $\de_* = \bigcup_{\delta < 2^\lambda} D_\delta$. 

Let us check that $\de_*$ is indeed a perfect ultrafilter.  If $\bar{\mb}$ satisfies condition \ref{d:perfect-0}(A), let $\uu$ 
be as there, and let $\delta = \beta + 1$ be an ordinal $< 2^\lambda$ such that $\bar{\mb}_\beta = \bar{\mb}$ and 
$\uu \subseteq \ba_{\beta, \mu, \theta}$, which is possible as we listed each sequence cofinally often. Then 
since $D_\beta$ was an ultrafilter, $\de_* \rstr \ba_{\beta, \mu, \theta} = \de_\beta$ so at stage $\delta$ 
condition (*) of the inductive hypothesis will be activated and 
we will have ensured that $\bar{\mb}$ has a multiplicative refinement in $\de_*$.  
\end{proof}

\vspace{5mm}

\section*{Appendix: non-saturation for $T_{k+1,k}$ and uncountable $\theta$}

\setcounter{section}{10}
\setcounter{equation}{0}

In this Appendix we update the non-saturation result from \cite{MiSh:1050} to allow for possibly uncountable $\theta$. The proof is the same. It has 
just been slightly rewritten for readability, since it seems a good occasion to call attention to Question \ref{q:qn}. 

\begin{theorem}[\cite{MiSh:1050}, Claim 5.1 for possibly uncountable $\theta$] \label{theta-22}
Suppose that:  
\begin{enumerate}
\item for  
integers $2 \leq k < \ell$, $\theta \geq \aleph_0$, and e.g. $\mu = 2^{\aleph_0}$, $\lambda = \mu^{+\ell}$, 
\\ or just: 
 $(\lambda, k, \mu^+) \rightarrow k+1$ in the sense of \cite{MiSh:1050} Notation 1.2 
\item $\ba = \ba^1_{2^\lambda, \mu, \theta}$   [for $\theta$ possibly uncountable] 
\item $\de_*$ is an ultrafilter on $\ba$
\item $T = T_{k+1,k}$ 
\end{enumerate}
Then $\de_*$ is not $(\lambda, T)$-moral. 
\end{theorem}

\begin{proof}  
Fix the objects given in the statement of the theorem. 
We'll often use the notation $T_{n,k}$ instead of $T_{k+1,k}$, but $n=k+1$ seems necessary in this proof, as we will point out.  
We will use \cite{MiSh:1050} Claim 1.6, which when applied to $(\lambda, k, \mu^+)$ gives: 
\begin{equation} \label{claim16}
\mbox{ there is a model $M \models T_{n,k}$ such that: }
\end{equation}

\begin{quotation}
\noindent $M$ has size $\geq \lambda$, and there are $\lambda$ elements of its domain $\langle b_\alpha : \alpha < \lambda \rangle$ such that if 
we let $\mcp = \{ w \in [\lambda]^n ~: ~ (\forall u \in [w]^{k+1})(~ M \models R(\bar{b}_u)) ~\}$ denote the indices for 
near-forbidden configurations,\footnote{recall that in $T_{n,k}$ the forbidden configuration is a set of $n+1$ vertices of which every $k+1$ form an edge.  So if $\{ b_\alpha : \alpha \in w \}$ is near-forbidden, then $\{ R(x,\bar{b}_v) : v \in [w]^k \}$ is not a type.} then for any $F: [\lambda]^k \rightarrow [\lambda]^{\leq \mu}$ such that $u \subseteq F(u)$ for 
all $u \in \dom(F)$,  there exists $w \in \mcp$ such that $(\forall v \in [w]^k) (w \not\subseteq F(v))$. 
\end{quotation}
Informally, the conclusion is that for any such $F$, some near-forbidden configuration escapes the control of its $k$-element subsets. 

The strategy will be to build a possibility pattern that has no multiplicative refinement. 
Fix a sequence of ordinals $\langle \alpha_w : w \in \mcp \rangle$, each $<2^\lambda$, with no repetitions. For each $w \in \mcp$, 
fix a function $g_w \in \fin_{\mu, \theta}(2^\lambda)$ such that $\dom(g_w) = \{ \alpha_w \}$ and $\mx_{g_w} = 0 \mod \de_*$. 

Let 
\begin{equation} \label{v-list} 
\langle v_\alpha : \alpha < \lambda \rangle
\end{equation} list $[\lambda]^k$ without repetition.\footnote{We are aiming at a type in the ultrapower of the 
form $\{  R(x,\bar{a}_{v_\alpha}) : \alpha < \lambda \}$. Since $u,v,w$ are used for sets of indices, we use $s \in \Omega$ for a finite set of 
formulas in the type.}  Let $\Omega = [\lambda]^{<\aleph_0}$ and for each $s \in \Omega$, let\footnote{Note that the condition in (\ref{eq17}) is set to avoid $\mx_{g_w}$ only if all $k$-element subsets of $w$ occur as $v_\beta$ for some $\beta \in s$. It's not enough that each element of $w$ occurs 
in some $v_\beta$. For example, in the tetrahedron-free three-hypergraph, 
if $w = \{ 1, 2, 3 \}$ and $\{ v_\beta : \beta \in s \} = \{ \{ 1,2\}, \{ 2, 3 \}, \{ 1, 3 \} \}$, then $\mb_s \cap \mx_{g_w} = 0$, but not if 
$\{ v_\beta : \beta \in s \} = \{ \{ 1, 2 \}, \{ 1, 3 \} \}$.} 
\begin{equation} \label{eq17}
\mb_s = 1_\ba - \bigcup \{ \mx_{g_w} : w \in \mcp \mbox{ and } [w]^k \subseteq \{ v_\beta : \beta \in s \} \}.
\end{equation}
In order to check that $\bar{\mb} = \langle \mb_s : s \in \Omega \rangle$ is really a representation of some type in the ultrapower, it will  
suffice by compactness to argue as follows (since we may always choose the index model to be a $\lambda^+$-saturated elementary 
extension of $M$). 
First note that for each finite $s \subseteq \lambda$, each nonzero $\mc \in \ba$, $\mc$ induces a partition of 
$\{ \mb_{s^\prime} : s^\prime \subseteq s \}$ according to whether $\mc \cap \mb_{s^\prime} = 0$ or $\mc \cap \mb_{s^\prime} > 0$. 
By shrinking $\mc$ if necessary, we may assume $\mc$ induces a partition of 
$\{ \mb_{s^\prime} : s^\prime \subseteq s \}$ according to whether $\mc \cap \mb_{s^\prime} = 0$ or $\mc \leq \mb_{s^\prime}$.  
Let $\vrt(s) := \bigcup \{ v_\beta : \beta \in s \}$ be the set of indices for all elements mentioned in formulas in $s$. Finally, by 
shrinking $\mc$ if necessary, we  
may also assume that either $\mc \leq \mx_{g_w}$ or $\mc \leq 1 - \mx_{g_w}$ for every $w \in \mcp$ such that $w \subseteq \vrt(s)$.
It suffices to show that for each such finite $s$ and nonzero $\mc$, we may choose elements 
$\{ b^\prime_{\beta} : \beta \in \vrt(s) \}$ in $M$ such that 
\begin{equation} \label{eq55}
M \models (\exists x) (\bigwedge_{\beta \in s^\prime} R(x,\bar{b}^\prime_{v_\beta}) \mbox{ when } \mc \leq \mb_{s^\prime} 
\end{equation} 
and 
\begin{equation} \label{eq56}
M \models \neg (\exists x) (\bigwedge_{\beta \in s^\prime} R(x,\bar{b}^\prime_{v_\beta}) \mbox{ when }\mc \cap \mb_{s^\prime} = 0.
\end{equation}

Consider the elements $\{ b_\beta : \beta \in \vrt(s) \}$. In $M$, this set may have some edges on it. Informally, what we will do is for each 
$w \subseteq \vrt(s)$ such that $w \in \mcp$, we 
remove the edge on $\{ b_\beta : \beta \in w \}$ if and only if $\mc \leq 1 - \mx_{g_w}$. Formally, we choose a 
set of distinct elements $\{ b^\prime_\beta : \beta \in \vrt(s) \}$ of $M$ such that $R(b^\prime_{\beta_0}, \dots, b^\prime_{\beta_{k}})$ 
if and only if $\{ \beta_0, \dots, \beta_k \} \in \mcp$ and $\mc \leq \mx_{g_w}$.
 
Let's check that (\ref{eq55}) and (\ref{eq56}) are satisfied. 

If $\mc \cap \mb_{s^\prime} = 0$, then there is some $w$ such that $w \subseteq \vrt(s)$, 
$w \in \mcp$, $[w]^k \subseteq \{ v_\beta : \beta \in s \}$ and $\mc \leq \mx_{g_w}$.  
[Suppose not. If there were no $w \subseteq \vrt(s^\prime)$ such that $w \in \mcp$ and $[w]^k \subseteq \{ v_\beta : \beta \in s^\prime \}$, 
then by definition in (\ref{eq17}),  $\mb_{s^\prime} = 1_\ba$, so we contradict $\mc > 0$. So there must be some such $w$. 
Let $w_0, \dots, w_i$ be a list of all such $w$. Then again by (\ref{eq17}), $\bigcup_{j \leq i} \mx_{g_{w_j}} = 1 - \mb_{s^\prime}$, so 
it must be that $\mc \leq \mx_{g_{w_j}}$ for some $j \leq i$.]  By construction, $M \models R(\bar{b}^\prime_w)$, so $\{ R(x,\bar{b}_{v_\beta}) : \beta \in s^\prime \} \supseteq \{ R(x,\bar{b}_{v_\beta} : v_\beta \in [w]^k \}$ is indeed inconsistent. 

If $\mc \leq \mb_{s^\prime}$, then for any $w$ such that 
$w \subseteq \vrt(s)$, 
$w \in \mcp$, and $[w]^k \subseteq \{ v_\beta : \beta \in s \}$, we must also have (by definition of $\mb_{s^\prime}$ in 
(\ref{eq17}) that $\mc \leq \mx_{g_w}$.  But in this case we removed the edge on $\{ b^\prime_\beta : \beta \in w \}$. 
Since this is true for all relevant $w$, $\{ R(x,\bar{b}^\prime_{v_\beta}) : \beta \in s^\prime \}$ is indeed consistent.

This shows that $\bar{\mb} = \langle \mb_s : s \in \Omega \rangle$ is indeed a possibility pattern, and it remains to show 
it has  no multiplicative refinement. Suppose for a contradiction that 
\[ \bar{\mb}^\prime = \langle \mb^\prime_s : s \in \Omega \rangle \]
were a multiplicative refinement of $\bar{\mb}$, i.e. $\bar{\mb}^\prime$ is a sequence of elements of $\de_*$ such that $s_1, s_2 \in \Omega$ implies $\mb^\prime_{s_1} \cap \mb^\prime_{s_2} = \mb^\prime_{s_1 \cap s_2}$ and for each $s \in \Omega$, $\mb^\prime_{s} \leq \ma_{s}$. 
As each $\mb^\prime_{\{\beta\}} \in \ba^+$, we may write 
$\mb^\prime_{\{\beta\}} = \bigcup \{ \mx_{h_{\beta, i}} : i < i(\beta) \leq \mu \}$ where $\langle h_{\beta,i} : i < i(\beta) \rangle$ is a set of pairwise 
inconsistent functions from $\fin_{\mu, \theta}(2^\lambda)$. Let $S_\beta = \bigcup \{ \dom(h_{\beta,i}) : i < i(\beta) \}$, so 
$S_\beta \subseteq 2^\lambda$ has cardinality $\leq \mu \cdot \theta = \mu$. 

\begin{subclaim} \label{subclaim-m}
Let $n = k+1$.
If $w \in \mcp$ then $\alpha_{w} \in \bigcup \{ S_\beta : v_\beta \in [w]^k \}$. 
\end{subclaim}

\begin{proof} 
Let $x = \{ \beta : v_\beta \in [w]^k \}$, which is a finite set since (\ref{v-list}) was without repetitions. As 
$\overline{\mb}^\prime$ is multiplicative, $\mb^\prime_x = \bigcap \{ \mb^\prime_\beta : \beta \in x \}$. 
Let $f \in \fin_{\mu, \theta}(2^\lambda)$ be such that $\mx_{f} \leq \mb^\prime_x$. 
Thus $\mx_f \leq \mb^\prime_{\{\beta\}}$ for each $\beta \in x$. 
Let $g = f \rstr \bigcup \{ S_\beta : \beta \in x \}$, noting that $g$ 
must be nonempty [indeed, if $\dom(f) \cap S_\beta = \emptyset$ for some $\beta \in x$, then necessarily
$\mx_f \cap (1-\mb^\prime_\beta) > 0$]. 
Then $\mx_{g} \leq \mb^\prime_{\{\beta\}}$ for all $\beta \in x$. This implies that $\mx_{g} \leq \mb^\prime_x \leq \mb_x$ 
because $\bar{\mb}^\prime$ refines $\bar{\mb}$. 
By definition in (\ref{eq17}), 
\[ \mb_x = 1_\ba - \bigcup \{ \mx_{g_u} : u \in \bad \mbox{ and } [u]^k \subseteq \{ v_\beta : \beta \in x \} \}. \]
So as $[w]^k \subseteq \{ v_\beta : \beta \in x \}$, necessarily $\mx_{g} \cap \mx_{g_{w}} = 0_\ba$. 
Since our Boolean algebra $\ba$ was generated freely, it must be that $\dom(g_{w}) \cap \dom(g) \neq \emptyset$, 
but $\dom(g_{w}) = \{ \alpha_{w} \}$. This shows that $\alpha_{w} \in \bigcup \{ S_\beta: v_\beta \in [w]^k \}$ as desired. 

\noindent\emph{This proves Subclaim \ref{subclaim-m}}. \hfill \end{proof}

Finally, define $ F: [\lambda]^k \rightarrow [\lambda]^{\leq \mu} $ 
by: if $v \in [\lambda]^k$ let $\beta$ be such that $v = v_\beta$, and let 
\[ F(v) = \bigcup \{ w \in [\lambda]^n : w \in \bad \mbox{ and } \alpha_w \in S_\beta \} \cup v. \]
Then  $F(v)$ is well defined, $F(v) \subseteq \lambda$, and $|F(v)| \leq \mu$ for $v \in [\lambda]^k$, 
since (\ref{v-list}) is without repetition and $|S_\beta| \leq \mu$.  
Now for all $w \in \bad$, there is $v = v_\beta \in [w]^k$ such that $\alpha_w \in S_\beta$. 
{Thus} $w \subseteq F(v)$. This shows that for all $w \in \bad$,  
\[ (\exists v \in [w]^k) ( w \subseteq F(v) ). \]
This is a contradiction to (\ref{claim16}), so $\overline{\mb}$ does not have a multiplicative refinement. 
Thus, $\de_*$ cannot be moral for $T_{k+1,k}$. This completes the proof. 
\end{proof}

\begin{qst}  \label{q:qn} 
Does Theorem \ref{theta-2} hold for $n>k+1$? 
\end{qst}

\end{document}